\documentclass{article}
\RequirePackage[OT1]{fontenc}
\RequirePackage{amsthm,amsmath}
\RequirePackage{hyperref}

\usepackage{graphicx}
\usepackage{multirow}
\usepackage{subfigure}
\usepackage{lipsum}
\usepackage{setspace}
\usepackage{geometry}
\newtheoremstyle{break}
  {\topsep}{\topsep}%
  {\itshape}{}%
  {\bfseries}{}%
  {\newline}{}%
\theoremstyle{break}

\newtheorem{theorem}{Theorem}
\newtheorem{remark}{Remark}
\newtheorem{definition}{Definition}

\newtheorem{lemma}{Lemma}
\newtheorem{corollary}{Corollary}

\newtheorem{algorithm}{Algorithm}
\newcommand{\RNum}[1]{\uppercase\expandafter{\romannumeral #1\relax}}
\begin{document}
\title{Bootstrap prediction intervals with asymptotic conditional validity and unconditional guarantees}

\author{ Yunyi Zhang\\
 Department of Mathematics \\
      University of California, San Diego  \\
          La Jolla, CA 92093-0112, USA
\and
 Dimitris N. Politis  \\
  Department of Mathematics and \\
  Halicioglu Data Science Institute \\
      University of California, San Diego  \\
          La Jolla, CA 92093-0112, USA  \\
             dpolitis@ucsd.edu }

\maketitle
\abstract{It can be argued that optimal prediction should take into account
 all  available data.
 Therefore,    to evaluate a prediction interval's performance one
should employ conditional coverage probability, conditioning on all  available
observations.
Focusing on a linear model, we derive the asymptotic distribution of the
 difference between the conditional coverage probability of a nominal prediction interval and the
conditional coverage probability of a prediction interval obtained via a  residual-based bootstrap. Applying this result, we show that a prediction interval generated by the residual-based bootstrap has approximately $50\%$ probability to yield conditional under-coverage. We then  develop a new bootstrap algorithm that generates a prediction interval that asymptotically
controls both the  conditional coverage probability as
well as  the possibility of conditional under-coverage. We complement the asymptotic
results with several finite-sample simulations.}

\section{Introduction}
Statistical inference comes in two flavors: explaining the world and predicting the future state of the world. To explain the world based on data, statisticians create models like linear regression and use data to fit the models. After doing that, they will gauge the goodness-of-fit, and assess the accuracy of estimation, e.g., via confidence intervals of the fitted model. Focusing on regression, the  literature is huge; to pick 3-4 papers, see Shao \cite{doi:10.1080/01621459.1996.10476934} on model selection, Xie and Huang \cite{xie2009}  or Liu and Yu \cite{liu2013} on model fitting, and Freedman \cite{freedman1981} on
statistical analysis.

Prediction is not a new topic in statistical inference; we refer to Geisser \cite{PrdInference} for a comprehensive introduction, or Politis \cite{Model-based} for a more recent exposition. Notably, prediction has seen a resurgence in the 21st century
with the advent of statistical learning; see Hastie et al. \cite{ESl} for an introduction. Similarly to  the aforementioned linear model procedure, statisticians use data to fit a
model that can yield a  predictor for future observations, and use prediction intervals to quantify uncertainty in the prediction; see e.g. Romano et al. \cite{NEURIPS2019_5103c358}.
Under a regression setting, there are several ways to construct a prediction interval.
The classical prediction interval was typically  obtained
under a Gaussian assumption on the errors; see Section \ref{EXAMM}
in that follows.
One of the earliest methods
foregoing  the restrictive normality assumption
employed the residual-based bootstrap; see   Stine \cite{doi:10.1080/01621459.1985.10478220}
and the references therein. More recent methods
include the  \textit{Model-free} (MF) bootstrap and the hybrid  \textit{Model-free/Model-based} (MF/MB) bootstrap of Politis \cite{Model-based}.

For all bootstrap methods, the aim is to  provide an asymptotically
valid prediction interval.
suppose $\Gamma$ is a prediction interval for the future observation $y_f$.
If  $Prob(y_f\in \Gamma)\approx 1-\alpha$ (where $\approx$ indicates an
asymptotic approximation), then $\Gamma$
is an asymptotically valid  $1-\alpha$ prediction interval for   $y_f$.
 On the other hand, if we wish to ensure that  $Prob(y_f\in \Gamma)\geq 1-\alpha$,
i.e.,  an  {\it unconditional lower-bound guarantee},
 then we may apply the conformal prediction idea of Shafer and Vovk \cite{ConPrd} and Vovk et al. \cite{AlgorithmWord}, which has been applied to several complex models, including non-parametric regression; see Lei and Wasserman \cite{doi:10.1111/rssb.12021},  Lei et al. \cite{doi:10.1080/01621459.2017.1307116},
Romano et al. \cite{NEURIPS2019_5103c358}, and Sesia and Cand\`es \cite{doi:10.1002/sta4.261}.

In the paper at hand, we assume a linear model and
  discuss how to construct
an asymptotically valid prediction interval in the context of conditional coverage that also possesses some unconditional guarantees as discussed above.
 To be more concrete, suppose we have an $n\times p$ design matrix $X$, independent and identically distributed residuals $\epsilon = (\epsilon_1,...,\epsilon_n)^T$, dependent variables $y = (y_1,...,y_n)^T$
where $y = X\beta + \epsilon$ and a fixed new regressor $x_f$
that is of interest. We would like to provide a $1-\alpha$ prediction interval $\Gamma = \Gamma(X, y, x_f)$ for the future observation $y_f = x_f^T\beta + \varepsilon$; here $\varepsilon$ is independent of $X,y$ and has the same distribution as $\epsilon_1$. The aforementioned bootstrap methods will ensure that $Prob\left(y_f\in\Gamma\right)\approx 1-\alpha$, but without a lower-bound guarantee.
On the other hand,
  the conformal prediction method yields an interval $\Gamma$ such that
  $Prob\left(y_f\in\Gamma\right)\geq 1-\alpha$, i.e., an {\it unconditional}
lower-bound guarantee. However,
we are more  interested in quantifying the performance of a prediction interval
in terms of its  conditional coverage probability $Prob\left(y_f\in\Gamma| y\right)$(or $Prob\left(y_f\in\Gamma| y, x_f, X\right)$ under random design).

 The reason for our interest comes from two aspects. On one hand, the conditional probability precisely describes how statisticians make prediction in practice. By  using the unconditional probability
\begin{equation}
Prob\left(y_f\in\Gamma\right) = E\left(Prob\left(y_f\in\Gamma|y\right)\right)
\label{COND}
\end{equation}
it is as if we assume that the statistician  has not observed $y$ before making the prediction.

Realistically, however, statisticians have observed $y$ and have fitted the model before they make predictions. Therefore, it is informative   to understand what happens to $y_f$
given our knowledge  
of all data (including $y$) rather than on ``average" among all possible $y$.

On the other hand, according to eq.  \eqref{COND}, analysis of the conditional probability is  a  more fundamental topic than the unconditional one. For example, if for any given $\xi>0$, $Prob\left(\{\vert Prob\left(y_f\in\Gamma|y\right) - (1-\alpha)\vert > \xi\}\right)\to 0$
as $n\to \infty$,  
then we can take the conditional expectation and have

\begin{equation}
\begin{aligned}
\vert Prob\left(y_f\in\Gamma\right) - (1-\alpha)\vert\leq E\left(\vert Prob\left(y_f\in\Gamma|y\right) - (1-\alpha)\vert\right)
\leq \xi + Prob\left(\{\vert Prob\left(y_f\in\Gamma|y\right) - (1-\alpha)\vert > \xi\}\right)
\end{aligned}
\end{equation}
which implies $Prob\left(y_f\in\Gamma\right)\to 1-\alpha$. 

Consequently, the
aforementioned performance goals of asymptotic validity and lower bound guarantee
should be recast in terms of conditional coverage.
Note, however, that  $Prob\left(y_f\in\Gamma|y\right)$ is a random variable itself -- see e.g. definition 1.3 in \c{C}inlar \cite{Stochastic}.
Hence, the performance goals are now stochastic, i.e., $Prob\left(y_f\in\Gamma|y\right)\to_p 1-\alpha$ and $Prob\left(y_f\in\Gamma|y\right)\geq 1-\alpha$ with a specific probability. Surprisingly, we can achieve these goals simultaneously through a careful re-design of our prediction intervals. Definition \ref{def1} in what follows describes our new performance
aim. Before stating it, however, we need to clarify our notation since our results hold true for
both fixed and random design. In the latter case, however, all probabilities and
expectations will be understood as being conditional on $X$; see
Definition \ref{defPP} below.

\begin{definition}
Consider the two cases:
\\ (a) {\bf Fixed design}, i.e., there is no randomness involved in the design matrix $X$ and the new regressor $x_f$. In this case,   we define $\mathbf{P}(\cdot) = Prob(\cdot)$, $\mathbf{P}^*(\cdot) = Prob(\cdot|y)$, $\mathbf{E}\cdot = E\cdot$, and $\mathbf{E}^*\cdot = E(\cdot|y)$.
\\ (b) {\bf Random design}, i.e., there is   randomness involved in  the design matrix $X$ (and possibly   in the new regressor $x_f$ as well).
In this case,   we define  $\mathbf{P}(\cdot) = Prob(\cdot|X,x_f)$, $\mathbf{P}^*(\cdot) = Prob(\cdot|y, X, x_f)$, $\mathbf{E}\cdot = E(\cdot|X, x_f)$, and $\mathbf{E}^*\cdot = E(\cdot|y,X,x_f)$.
Furthermore, convergences and probability statements will be understood to hold
almost surely in $X $ and $  x_f $.
\label{defPP}
\end{definition}

We can now state our new performance aims in general.

\begin{definition}[Prediction interval with unconditional guarantee]
Assume an  $n\times p$ design matrix $X$, independent and identically distributed residuals $\epsilon =(\epsilon_1,...,\epsilon_n)^T\in\mathbf{R}^n$, and that the dependent variables $y$ satisfy a linear model $y = X\beta + \epsilon$.
For a regressor $x_f\in\mathbf{R}^p$ and a potential future observation
$y_f$, we say that $\Gamma = \Gamma(X,y, x_f)$ is the $1-\alpha$ prediction interval with $1-\gamma$ unconditional guarantee if the following conditions hold true:

1. For any given $\xi>0$,
\begin{equation}
\mathbf{P}\left(\{\vert\mathbf{P}^*\left(y_f\in\Gamma\right) - (1-\alpha)\vert > \xi\}\right)\to 0
\end{equation}

2.
\begin{equation}
\mathbf{P}\left(\{\mathbf{P}^*(y_f\in\Gamma)\geq 1-\alpha\}\right)\to 1-\gamma
\label{ALLNEED}
\end{equation}
as $n\to \infty$; here, $\alpha,\gamma $ are constants in $(0,1)$. We call $1-\alpha$ the nominal (conditional) coverage probability and $1-\gamma$ the {\it guarantee level}.
\label{def1}
\end{definition}

Definition \ref{def1} does not provide a free lunch; to see why, we present a simulation and record quantiles of conditional coverage probabilities and guarantee levels for the aforementioned methods in table \ref{EXAMP}. The majority of them do not have a high guarantee level, which leads to the possibility of conditional under-coverage.

\begin{table}[htbp]
  \centering
  \caption{Quantiles of conditional coverage probabilities and guarantee levels of prediction intervals on the \textbf{Experiment model} (see section \ref{Numerics} and table \ref{tableMain}) with `Normal' errors; the nominal coverage probability is $95\%$. We use the R-package maintained by Tibshirani \cite{Rpackage} to perform conformal predictions.}
  \scriptsize
  \begin{tabular}{l l l l l l l}
  \hline\hline
  Sample size & Algorithm                        & \multicolumn{4}{l}{Quantiles of coverage probabilities} & Guarantee level\\
              &                                  &      $25\%$    & $45\%$    & $65\%$   &  $85\%$         &         \\
  100         & Residual bootstrap               &      $91.0\%$  & $92.8\%$  & $94.2\%$  &  $95.6\%$      & $23.6\%$\\
              & \textit{MF/MB} bootstrap         &      $95.5\%$  & $96.6\%$  & $97.4\%$  &  $98.3\%$      & $80.6\%$\\
              & Conformal prediction             &      $92.4\%$  & $94.2\%$  & $95.6\%$  &  $97.0\%$      & $44.2\%$\\
              & Split conformal prediction       &      $95.9\%$  & $97.7\%$  & $98.7\%$  &  $99.5\%$      & $80.7\%$\\
              & Jackknife conformal prediction   &      $93.2\%$  & $94.9\%$  & $96.1\%$  &  $97.4\%$      & $53.0\%$\\
  \hline
  400         & Residual bootstrap               &      $93.6\%$  & $94.4\%$  & $95.0\%$  &  $95.8\%$      & $36.6\%$\\
              & \textit{MF/MB} bootstrap         &      $94.6\%$  & $95.3\%$  & $95.9\%$  &  $96.6\%$      & $64.2\%$\\
              & Conformal prediction             &      $93.0\%$  & $93.9\%$  & $94.6\%$  &  $95.4\%$      & $24.9\%$\\
              & Split conformal prediction       &      $94.3\%$  & $95.3\%$  & $96.2\%$  &  $97.0\%$      & $62.8\%$\\
              & Jackknife conformal prediction   &      $94.2\%$  & $95.0\%$  & $95.5\%$  &  $96.2\%$      & $53.8\%$\\
  \hline
  1600         & Residual bootstrap              &      $94.3\%$  & $94.7\%$  & $95.2\%$  &  $95.7\%$      & $42.6\%$\\
              & \textit{MF/MB} bootstrap         &      $94.6\%$  & $95.0\%$  & $95.4\%$  &  $95.9\%$      & $54.4\%$\\
              & Conformal prediction             &      $93.3\%$  & $93.8\%$  & $94.4\%$  &  $94.9\%$      & $12.4\%$\\
              & Split conformal prediction       &      $94.5\%$  & $95.0\%$  & $95.4\%$  &  $96.0\%$      & $42.6\%$\\
              & Jackknife conformal prediction   &      $94.5\%$  & $94.9\%$  & $95.3\%$  &  $95.7\%$      & $54.4\%$\\
  \hline\hline
  \end{tabular}
  \label{EXAMP}
\end{table}

Our paper has two main contributions. On the one hand, it derives the Gaussian approximation for the difference between the conditional probability of a nominal prediction interval and the conditional probability of a prediction interval based on residual-based bootstrap. 
 In practice, bootstrap approximates the former by the latter, and the non-zero difference will make the   former deviate
 from $1-\alpha$. This   leads to the fact that the  \textit{residual-based bootstrap algorithm asymptotically has guarantee level of $50\%$}. On the other hand, we develop a
new method to construct a prediction interval satisfying definition \ref{def1} with arbitrarily chosen $\alpha,\gamma$.

We employ a simple example to illustrate why a classical prediction interval becomes problematic under the conditional coverage context in section \ref{EXAMM}. After that, we introduce the frequently used notations and assumptions in section \ref{chp12}. In section \ref{ch2},  we derive the Gaussian approximation result. In section \ref{chp3}, we develop the algorithm to construct the newly proposed prediction interval. We perform some simulations to illustrate the proposed algorithm's finite sample performance in section \ref{Numerics}, and provide some conclusions in section \ref{ch5}.

\section{An intuitive illustration in the Gaussian case}
\label{EXAMM}
For the sake of illustration, in this section only we suppose the residual $\epsilon_1$ has a normal distribution with mean $0$ and known variance $\sigma^2$.   Denote $d_\alpha$
the $\alpha$-th quantile and $\Phi(x)$ the cumulative distribution function of the standard normal distribution respectively, i.e., $d_\alpha = \Phi ^{-1} (\alpha)$; and adopt the notations $\mathbf{P},\mathbf{P}^*$ in definition \ref{defPP}.
If we do not care about the conditional coverage, we can define $\widehat{\beta} = (X^TX)^{-1}X^Ty$ and use the
normal distribution  $1-\alpha$ prediction interval $\mathcal{P}_1 = [x_f^T\widehat{\beta} + \sigma d_{\alpha / 2}\sqrt{1 + x_f^T(X^TX)^{-1}x_f} , x_f^T\widehat{\beta} + \sigma d_{1 - \alpha / 2}\sqrt{1 + x_f^T(X^TX)^{-1}x_f}]$ for the future response $y_f$. Since the random variable $y_f - x_f^T\widehat{\beta}$ has normal distribution with mean $0$ and variance $ \sigma^2(1 + x_f^T(X^TX)^{-1}x_f)$, it follows that
\begin{equation}
\mathbf{P}\left(y_f \in \mathcal{P}_1\right) = \mathbf{P}\left(d_{\alpha / 2}\leq \frac{y_f - x_f^T\widehat{\beta}}{\sigma\sqrt{1 + x_f^T(X^TX)^{-1}x_f}}\leq d_{1-\alpha / 2}\right) = 1 -\alpha.
\label{OOOSS}
\end{equation}
In other words,  $\mathcal{P}_1$ has precise unconditional coverage probability. However, if we take the conditional coverage into consideration, the random variable $y_f - x_f^T\widehat{\beta}|y$ (or $y_f-x_f^T\widehat{\beta}|y,x_f,X$ under random design) has normal distribution with mean $x^T_f\beta - x_f^T(X^TX)^{-1}X^Ty$ and variance $\sigma^2$. According to Taylor's theorem,
\begin{equation}
\begin{aligned}
\mathbf{P}^*\left(y_f\in P_1\right) = \mathbf{P}^*\left(d_{\alpha / 2}\leq \frac{y_f - x_f^T\widehat{\beta}}{\sigma\times \sqrt{1 + x_f^T(X^TX)^{-1}x_f}}\leq d_{1 - \alpha /2}\right)\\
= \Phi\left(\sqrt{1 + x_f^T(X^TX)^{-1}x_f}d_{1 - \alpha / 2} + \frac{x_f^T(X^TX)^{-1}X^T\epsilon}{\sigma}\right) - \Phi\left(\sqrt{1 + x_f^T(X^TX)^{-1}x_f}d_{\alpha / 2} + \frac{x_f^T(X^TX)^{-1}X^T\epsilon}{\sigma}\right)\\
\approx 1 - \alpha + \Phi^{'}(d_{1 - \alpha / 2})\times x_f^T(X^TX)^{-1}x_f \times d_{1-\alpha / 2} + \Phi^{''}(d_{1-\alpha / 2})\times(\frac{x_f^T(X^TX)^{-1}X^T\epsilon}{\sigma})^2
\end{aligned}
\label{JKLS}
\end{equation}
Since $\Phi^{''}(d_{1-\alpha / 2}) < 0$, $\frac{(x_f^T(X^TX)^{-1}X^T\epsilon)^2}{\sigma^2(x_f^T(X^TX)^{-1}x_f)}$ has $\chi^2_1$ distribution and $\Phi^{''}(x) = -x\Phi^{'}(x)$ for any $x$,
\begin{equation}
\mathbf{P}\left(\{\mathbf{P}^*\left(y_f\in \mathcal{P}_1\right)\geq 1-\alpha\}\right)\approx \mathbf{P}\left(\frac{(x^T_f(X^TX)^{-1}X^T\epsilon)^2}{\sigma^2 x_f^T(X^TX)^{-1}x_f}\leq \frac{\Phi^{'}(d_{1-\alpha / 2})\times d_{1-\alpha / 2}}{-\Phi^{''}(d_{1-\alpha / 2})}\right) \approx 0.683
\end{equation}
Therefore, the prediction interval $\mathcal{P}_1$ only has about $68\%$ guarantee level.

However, it is possible to find a prediction interval with a desired guarantee level,
say $1-\gamma$. We define $C_{1-\gamma}$ as the $1-\gamma$ quantile of a $\chi^2_1$ distribution, and let $c_{1-\gamma} = -\Phi^{''}(d_{1 - \alpha / 2})x_f^T(X^TX)^{-1}x_f\times C_{1-\gamma} / (2\Phi^{'}(d_{1-\alpha / 2}))>0$. We construct the prediction interval $\mathcal{P}_2 = [x_f^T\widehat{\beta} + \sigma\times (d_{\alpha / 2} - c_{1-\gamma}),\ x_f^T\widehat{\beta} + \sigma\times (d_{1 - \alpha / 2} + c_{1-\gamma})]$. We can now compute
\begin{equation}
\begin{aligned}
\mathbf{P}^*\left(y_f \in \mathcal{P}_2\right)= \mathbf{P}^*\left(d_{\alpha / 2} - c_{1-\gamma}\leq \frac{y_f - x_f^T\widehat{\beta}}{\sigma}\leq d_{1 - \alpha / 2} + c_{1-\gamma}\right)\\
= \Phi\left(d_{1 - \alpha / 2} + c_{1-\gamma} + \frac{x_f^T(X^TX)^{-1}X^T\epsilon}{\sigma}\right) - \Phi\left(d_{\alpha / 2} - c_{1-\gamma} + \frac{x_f^T(X^TX)^{-1}X^T\epsilon}{\sigma}\right)\\
\approx 1 - \alpha + 2\Phi^{'}(d_{1-\alpha / 2})c_{1-\gamma} + \Phi^{''}(d_{1-\alpha / 2})\times(\frac{x_f^T(X^TX)^{-1}X^T\epsilon}{\sigma})^2\\
\mbox{which implies that } \ \  \mathbf{P}\left(\{\mathbf{P}^*\left(y_f\in \mathcal{P}_2\right)\geq 1-\alpha\}\right)\\
\approx \mathbf{P}\left(-\Phi^{''}(d_{1-\alpha / 2})\frac{(x_f^T(X^TX)^{-1}X^T\epsilon)^2}{\sigma^2}\leq -\Phi^{''}(d_{1-\alpha / 2})x_f^T(X^TX)^{-1}x_fC_{1-\gamma}\right) = 1 - \gamma .
\end{aligned}
\label{JXT}
\end{equation}
Hence, prediction interval $\mathcal{P}_2 $ has guarantee level $1-\gamma$.
Note that since $c_{1-\gamma}$ has order $O(1/n)$, this correction does
not significantly enlarge the length of the prediction interval.

  In general, however,  the marginal distribution of the errors can not be
assumed to be normal. As a consequence,  we need to use resampling to find
a  satisfactory correction; this will be the subject of the following sections.

\section{Preliminary notions}
\label{chp12}
For the remainder of the paper, we  revert to the general setup:  a   design matrix $X$
(assumed to have full-rank), the dependent variable $y$ satisfying the linear model $y = X\beta + \epsilon$ with respect to the i.i.d. errors $\epsilon = (\epsilon_1,...,\epsilon_n)^T$; here,  $\epsilon_1$ has mean zero,
 variance $\sigma^2$, and  cumulative distribution function denoted by $F$. We denote $X^T = (x_1,...,x_n)$, $x_i = (x_{i1},...,x_{ip})^T\in\mathbf{R}^p, $ 
 and define the estimated residual $\widehat{\epsilon}^{'} = (\widehat{\epsilon}^{'}_1,...,\widehat{\epsilon}^{'}_n)^T$, centered estimated residual $\widehat{\epsilon} = (\widehat{\epsilon}_1,...,\widehat{\epsilon}_n)^T$ and residual empirical process $\widehat{F}(x)$ for any $x\in\mathbf{R}$ respectively as
\begin{equation}
\begin{aligned}
\widehat{\epsilon}_i^{'} = y_i-x_i^T\widehat{\beta} = \epsilon_i-x_i^T(\widehat{\beta}-\beta)\\
\widehat{\epsilon}_i=\widehat{\epsilon}_i^{'}-\frac{1}{n}\sum_{i=1}^n\widehat{\epsilon}^{'}_i\\
\widehat{F}(x) = \frac{1}{n}\sum_{i=1}^n\mathbf{1}_{\widehat{\epsilon}_i\leq x} .
\label{DefResi}
\end{aligned}
\end{equation}
Denote
$\widehat{\beta} = (X^TX)^{-1}X^Ty$ as the least square estimator of parameter
vector  $\beta$. We also denote $\widehat{\lambda} = \frac{1}{n}\sum_{i=1}^n\widehat{\epsilon}_i^{'} = \frac{1}{n}\sum_{i=1}^n\epsilon_i-\overline{x}_n^T(\widehat{\beta}-\beta)$, and $\overline{x}_n = \frac{1}{n}\sum_{i=1}^n x_i$. From \eqref{DefResi}, we have
\begin{equation}
\int xd\widehat{F} = \frac{1}{n}\sum_{i=1}^n\widehat{\epsilon}_i = 0,\ \  \widehat{\sigma}^2 = \int x^2d\widehat{F} = \frac{1}{n}\sum_{i=1}^n\widehat{\epsilon}_i^2 .
\label{sigmaHat}
\end{equation}

We denote $\mathbf{D} = \mathbf{D}[0,1]$   the space of \textit{c\`adl\`ag} functions on $[0,1]$ with Skorohod topology--see chapter 3 of Billingsley \cite{billing}.

To derive our results, we require the following assumptions:

\begin{enumerate}

\item  $\epsilon_1$'s distribution is absolutely continuous with respect to Lebesgue measure. $F$ is second order continuous differentiable and $\sup_{x\in\mathbf{R}}\vert F^{''}(x)\vert<\infty$, $\mathbf{E}\epsilon_1 = 0$, $\mathbf{E}\vert\epsilon_1\vert^4<\infty$. The new regressor $x_f\in\mathbf{R}^p$ and the new dependent variable $y_f$ satisfy $y_f = x_f^T\beta + \varepsilon$. $\varepsilon$ is independent of $\epsilon$ and has the same distribution as $\epsilon_1$.

\item  One of the two following conditions holds true:

\quad 2.1) {\bf Fixed design:} $X$ and $x_f$ are fixed, i.e., non-random.

\quad 2.2) {\bf  Random design:} $X$ and $x_f$ are random. However, $x_f$ is independent of $\epsilon,\varepsilon$; and $X$ is independent of $\epsilon,\varepsilon, x_f$.

\item  $X^TX$ is invertible for $\forall n\geq p$ and $\lim_{n\to\infty} \frac{X^TX}{n} = A$, $\lim_{n\to\infty}\overline{x}_n = b$; here $A$ is an invertible matrix and $b\in\mathbf{R}^p$. Besides, there exists a constant $M > 0$ such that $\Vert x_i\Vert_2\leq M$ for $i=1,2,...,n$ and $\Vert x_f\Vert_2\leq M$. $\Vert.\Vert_2$ denotes the Euclidean norm.

We define $H(x) = \mathbf{E}\epsilon_1\mathbf{1}_{\epsilon_1\leq x}$ and for $\forall x,z\in\mathbf{R}$,
\begin{equation}
\mathcal{V}(x,z) = \sigma^2F^{'}(x)F^{'}(z)\left(x_f^TA^{-1}x_f+1-2x_f^TA^{-1}b\right) - (F^{'}(x)H(z)+F^{'}(z)H(x))(x_f^TA^{-1}b - 1) + F(\min(x,z)) - F(x)F(z)
\label{VarFun}
\end{equation}
We also define $\mathcal{U}(x) = \mathcal{V}(x,x) + \mathcal{V}(-x,-x) - 2\mathcal{V}(x,-x)$.

\item $F^{'}(x)>0,\forall x\in\mathbf{R}$, and $\mathcal{U}(x)> 0, \forall 0< x<\infty$.
\end{enumerate}

\noindent
For a function $f$ and a point $x\in\mathbf{R}$, we define $f^{-}(x) = \lim_{y\to x,y<x}f(y)$ if this limit exists. Note that $f\in\mathbf{D}$ implies that $ f^{-}(x)$ exists for $\forall x\in(0,1)$. As in section 1.1.4 of Politis et al. \cite{subsampling}, for any   $0<\alpha<1$, we define the  $\alpha$ quantile of a  cumulative distribution function $f$ as
\begin{equation}
c_{\alpha} = \inf\{x\in\mathbf{R}|f(x)\geq \alpha\} .
\label{HJH}
\end{equation}

The meaning of notations $\mathbf{P}, \mathbf{P}^*, \mathbf{E}, \mathbf{E}^*$ is presented in definition \ref{defPP}.
The symbol $\to$ represents convergence in $\mathbf{R}$, and $\to_{\mathcal{L}}$ represents convergence in distribution. Without being specified, the convergence assumes the sample size $n\to\infty$. $\Phi$ represents the cumulative distribution function of the standard normal distribution. In the case of random design, the convergence results hold true for almost sure $X$ and $x_f$.

\begin{remark}
(a) We centered the estimated residuals $\widehat{\epsilon}^{'}$ in eq. \eqref{DefResi}, but if the design matrix $X$ has a column of ones, then summation of the estimated residuals will be $0$ exactly, and  re-centering is superfluous.
\\ (b) In the case of random design, we assume assumption 3. and 4. happen for almost sure $X$ and $x_f$.
\end{remark}

\section{Gaussian approximation in bootstrap prediction}
\label{ch2}

Residual-based bootstrap has been widely used in interval prediction for various models, such as Thombs and Schucany \cite{doi:10.1080/01621459.1990.10476225},  and Li and Politis \cite{PAN2016467}. Stine \cite{doi:10.1080/01621459.1985.10478220} introduced a residual-based bootstrap algorithm for prediction, but this algorithm is typically characterized by finite sample undercoverage; see  Li and Politis \cite{PAN20161}. To alleviate the finite-sample undercoverage, Politis \cite{Model-based} proposed the {\it Model-free/Model-Based (MF/MB) bootstrap}, that
resamples  the  {\it predictive} residuals $\widehat{r} = (\widehat{r}_1,...,\widehat{r}_n)^T$
  instead of the usual fitted residuals.
The predictive  residuals are sometimes called the `leave-one-out' residuals, and are defined as:
\begin{equation}
\begin{aligned}
\widehat{r}_i^{'} = y_i - x_i^T(X^T_{-i}X_{-i})^{-1}X^T_{-i}y_{-i},\ \
\widehat{r}_i = \widehat{r}_i^{'} - \frac{1}{n}\sum_{i=1}^n\widehat{r}_i^{'},\ i=1,2,...,n\\
\text{where $X_{-i}$ and $y_{-i}$ are  the design matrix $X$ and the dependent variable
vector $y$ respectively, having left out the $i$th row.}
\end{aligned}
\label{PRD}
\end{equation}

For concreteness, the algorithms are as follows:
\begin{algorithm}[Residual-based bootstrap]
\textbf{Input: } Design matrix $X$ and dependent variable data vector $y$
satisfying $y = X\beta + \epsilon$, the new regression vector $x_f$
of interest, number of bootstrap replicates $B$, nominal coverage probability $1-\alpha$

1) Calculate statistics $\widehat{\beta} = (X^TX)^{-1}X^Ty$ and $\widehat{\epsilon} = (\widehat{\epsilon}_1,...,\widehat{\epsilon}_n)^T$ as in eq.  \eqref{DefResi}.

2) Generate i.i.d. residuals $\epsilon^* = (\epsilon_{1}^*,...,\epsilon^*_{n})^T$ and $\varepsilon^*$ by drawing from $\widehat{\epsilon}_1,...,\widehat{\epsilon}_n$ with replacement, then calculate $y^* = X\widehat{\beta} + \epsilon^*$ and $y^*_f = x_f^T\widehat{\beta} + \varepsilon^*$. Re-estimate $\widehat{\beta}^* = (X^TX)^{-1}X^Ty^*$ and calculate the prediction root $\delta^*_b = y^*_{f} - x_f^T\widehat{\beta}^*$

3) Repeat 2) for $b = 1,2,...,B$, and calculate the $1-\alpha$ (unadjusted) sample quantile $\widehat{c}^*_{1-\alpha}$ of $\vert\delta^*_b\vert$, $b=1,2,...,B$.

4) The prediction interval of $y_f$ is given by $\left\{y_f\Big|\vert y_f - x^T_f\widehat{\beta}\vert\leq \widehat{c}^*_{1-\alpha}\right\}$
\label{alg1}
\end{algorithm}

\begin{remark}
If we replace $\widehat{\epsilon}$ by   $\widehat{r}$ in algorithm \ref{alg1},
  we then obtain  the \textit{MF/MB bootstrap} algorithm.
\end{remark}

The Glivenko - Cantelli theorem ensures the empirical process of the bootstrapped prediction root $y_f^* - x_f^T\widehat{\beta}^*$ converges to $\mathbf{P}^*\left( y_f^* - x_f^T\widehat{\beta}^*\leq x\right)$ for any $x\in\mathbf{R}$ $\mathbf{P}^*$ almost surely as $B\to\infty$. Therefore, the residual-based bootstrap approximates the unobservable function $\mathbf{P}^*(\vert y_f-x_f^T\widehat{\beta}\vert\leq x)$ by $\mathbf{P}^*\left(\vert y_f^* - x_f^T\widehat{\beta}^*\vert\leq x\right)$, and estimates the latter distribution by the bootstrapped prediction root's empirical process; see Politis et al. \cite{subsampling}. This approximation introduces an error; we will now derive  the asymptotic distribution of the error
\begin{equation}
\mathcal{S}(x) = \sqrt{n}\left(\mathbf{P}^*(\vert y_f-x_f^T\widehat{\beta}\vert\leq x) - \mathbf{P}^*(\vert y_f^* - x_f^T\widehat{\beta}^*\vert\leq x)\right)
\label{Sxxxx}
\end{equation}
where
$y_f^*, \widehat{\beta}^*$ are defined in algorithm \ref{alg1}.
We refer to Bickel and Freedman \cite{10.2307/2240410} and Politis et al.\cite{subsampling} for related results.

For any given positive integer $m$, we define a Gaussian process $\mathcal{M}_m(x),x\in[0,1]$ in $\mathbf{D}$ with
\begin{equation}
\begin{aligned}
\mathbf{E}\mathcal{M}_m(x) = 0,\ \  \mathbf{E}\mathcal{M}_m(x)\mathcal{M}_m(z) =\mathcal{V}(2mx-m,2mz-m),\ \  \forall x,z\in[0,1]
\end{aligned}
\label{stopro}
\end{equation}
and $\mathcal{M}_m$ has continuous sample paths almost surely. $\mathcal{V}$ is defined in \eqref{VarFun}. Existence of $\mathcal{M}_m$ for $\forall m$ is proved in appendix \ref{lemmaS}.

\begin{theorem}
Suppose assumptions 1. to 4. hold true. Then,  for any given positive integer $0<m<\infty$,
\begin{equation}
\widetilde{M}_{m}(x) = \sqrt{n}F^{'}(x^{'})\left(x_f^T(X^TX)^{-1}X^T\epsilon-\frac{1}{n}\sum_{i=1}^n\epsilon_i\right) -\frac{1}{\sqrt{n}}\sum_{i=1}^n\left(\mathbf{1}_{\epsilon_i\leq x^{'}}-F(x^{'})\right)\to_{\mathcal{L}}\mathcal{M}_m(x)
\label{AsymptoticM}
\end{equation}
as $n\to\infty$ under Skohord topology in $\mathbf{D}$; here $x^{'} = 2mx-m$. Besides, for any given positive numbers $0<r<s<\infty$,
\begin{equation}
\begin{aligned}
\sup_{x\in [r,s]}\sup_{y\in\mathbf{R}}\vert \mathbf{P}\left(\mathcal{S}(x)\leq y\right)  - \Phi\left(\frac{y}{\sqrt{\mathcal{U}(x)}}\right)\vert\to 0 .
\label{KeyThm}
\end{aligned}
\end{equation}
\label{theoPd}
\end{theorem}
The variance of $\mathcal{M}_m\left(\frac{x+m}{2m}\right) - \mathcal{M}_m\left(\frac{-x+m}{2m}\right)$ is $\mathcal{U}(x) = \mathcal{V}(x,x) + \mathcal{V}(-x,-x) - 2\mathcal{V}(x,-x)$  when $ x\in\mathbf{R}, m > \vert x\vert + 1$ . Hence, theorem \ref{theoPd} implies $\mathcal{S}(x)$ 
has  asymptotically   the same distribution as $\widetilde{M}_m\left(\frac{x + m}{2m}\right) - \widetilde{M}_m^{-}\left(\frac{-x+m}{2m}\right)$,
with arbitrarily chosen $m > |x| + 1$, since $m$ does not influence $\mathcal{U}$.

In the conditional coverage context, an application of theorem \ref{theoPd} is to calculate a prediction interval's guarantee level. For example, by choosing    $y = 0$,  and $x = c^*_{1-\alpha}$ which denotes the $1-\alpha$ quantile of the distribution $\mathbf{P}^*(\vert y_f^* - x_f^T\widehat{\beta}^*\vert\leq x)$, we have the following corollary
\begin{corollary}
Under assumptions 1. to 4., the prediction interval generated by residual-based bootstrap has an asymptotically $50\%$ guarantee level.
\label{CORO1}
\end{corollary}

Alternatively, we could  choose $y = d_\gamma$, the $\gamma$ quantile of the standard normal distribution, and  $x=
 c^*_{1-\alpha-d_\gamma {\sqrt{\mathcal{U}(c^*_{1-\alpha})} }/{\sqrt{n}}}$.

Since $\mathcal{U}$ is continuous, theorem \ref{theoPd} implies the event
$
\mathbf{P}^*(\vert y_f-x_f^T\widehat{\beta}\vert\leq
c^*_{1-\alpha-d_\gamma {\sqrt{\mathcal{U}(c^*_{1-\alpha})} }/{\sqrt{n}}} ) - (1 - \alpha) \geq 0$, which is equivalent to the event
\begin{equation}
\begin{aligned}
\sqrt{n}\left(\mathbf{P}^*(\vert y_f-x_f^T\widehat{\beta}\vert\leq c^*_{1-\alpha-d_\gamma {\sqrt{\mathcal{U}(c^*_{1-\alpha})} }/{\sqrt{n}}}) - (1-\alpha- \frac{d_\gamma\sqrt{\mathcal{U}(c^*_{1-\alpha})} }{\sqrt{n}})\right)\geq d_\gamma \sqrt{\mathcal{U}(c^*_{1-\alpha})}
\end{aligned}
\label{findPI}
\end{equation}
asymptotically has unconditional probability $1-\gamma$. In other words, the prediction interval $\{y_f\Big|\vert y_f-x_f^T\widehat{\beta}\vert\leq c^*_{1-\alpha-d_\gamma {\sqrt{\mathcal{U}(c^*_{1-\alpha})} }/{\sqrt{n}}}\}$ has asymptotic guarantee level $1-\gamma$. Section \ref{chp3} adopts this idea. However, since estimating $\mathcal{U}$ is difficult,
we will focus on finding  a prediction interval with unconditional guarantee through resampling.

\section{Bootstrap prediction interval with unconditional guarantee}
\label{chp3}
Residual-based bootstrap and \textit{MF/MB bootstrap} generate asymptotically valid
prediction intervals. However, the statistician  cannot adjust those prediction intervals' guarantee level. This section proposes
two new variations on these bootstrap methods, namely the  \textit{Residual bootstrap with unconditional guarantee (RBUG)}
and the \textit{Predictive residual bootstrap with unconditional guarantee (PRBUG)}, that
 maintain  the asymptotic validity but also  allows us to choose the prediction interval's guarantee level.

\begin{algorithm}[\textit{RBUG}/\textit{PRBUG}]

\textbf{Input:} Design matrix $X$ and dependent variable data vector $y$
satisfying $y = X\beta + \epsilon$, the new regression vector $x_f$
of interest, and number of bootstrap replicates $B$, number of replicates to find quantile's adjustment $\mathcal{B}_1$, number of Monte Carlo integration steps $\mathcal{B}_2$, nominal coverage probability $1-\alpha$, and nominal guarantee level $1-\gamma$

\textbf{Note:} For \textit{RBUG}, we define $\widehat{\tau} = (\widehat{\tau}_1,...,\widehat{\tau}_n)^T= \widehat{\epsilon}$ as  in \eqref{DefResi},  while for \textit{PRBUG}, we define $\widehat{\tau} = \widehat{r}$ as  in \eqref{PRD}.

\textbf{Calculate an unadjusted sample quantile}

1) Calculate the statistics $\widehat{\beta} = (X^TX)^{-1}X^Ty$ and $\widehat{\tau}$.

2) Generate i.i.d. residuals $\epsilon^* = (\epsilon_{1}^*,...,\epsilon^*_{n})^T$ and $\varepsilon^*$ by drawing from $\widehat{\tau}_1,...,\widehat{\tau}_n$ with replacement; calculate $y^* = X\widehat{\beta} + \epsilon^*$, $y^*_f = x_f^T\widehat{\beta} + \varepsilon^*$, and $\widehat{\beta}^* = (X^TX)^{-1}X^Ty^*$; derive the prediction root $\delta^*_b = y^*_{f} - x_f^T\widehat{\beta}^*$.

3) Repeat 2) for $b = 1,2,...,B$, and calculate the $1-\alpha$ unadjusted sample quantile
(denoted as  $\widehat{c}^*_{1-\alpha}$) of $\vert\delta^*_b\vert$, $b=1,2,...,B$.

\textbf{Simulate eq. \eqref{AsymptoticM} to find the quantile adjustment}

4) Generate i.i.d. $e^* = (e^*_{1},...,e^*_{n})^T$ by drawing from $\widehat{\tau}_1,...,\widehat{\tau}_n$ with replacement, then derive $y^\dagger = X\widehat{\beta} + e^*$, $\widehat{\beta}^\dagger = (X^TX)^{-1}X^Ty^\dagger$

5) Generate i.i.d. $\varepsilon^*_1,...,\varepsilon^*_{\mathcal{B}_2}$ by drawing from $\widehat{\tau}_1,...,\widehat{\tau}_n$ with replacement. For $b_2 = 1,2,...,\mathcal{B}_2$, derive $\zeta^*_{b_2} = x_f^T\widehat{\beta} + \varepsilon^*_{b_2} - x_f^T\widehat{\beta}^\dagger + \frac{1}{n}\sum_{i=1}^ne^*_{i}$.
Define
\begin{equation}
p^*_{b_1} = \sqrt{n}\left(\frac{1}{n}\sum_{i=1}^n\mathbf{1}_{\vert e^*_{i}\vert\leq \widehat{c}^*_{1-\alpha}}- \frac{1}{\mathcal{B}_2}\sum_{b_2=1}^{\mathcal{B}_2}\mathbf{1}_{\vert\zeta^*_{b_2}\vert\leq \widehat{c}^*_{1-\alpha}} \right)
\label{Eq29}
\end{equation}

6) Repeat step 4) to 5) for $b_1=1,2,...,\mathcal{B}_1$, then calculate the $1-\gamma$ sample quantile (denoted as $\widehat{d}^*_{1-\gamma}$)  of $p^*_{b_1}, b_1=1,2,...,\mathcal{B}_1$.

\textbf{Construct the prediction interval}

7) Calculate $\widehat{c}^*_{1-\alpha + \widehat{d}^*_{1-\gamma}/\sqrt{n}}$, the $1-\alpha + \widehat{d}^*_{1-\gamma}/\sqrt{n}$ sample quantile of $\vert\delta^*_b\vert,\ b=1,2,...,B$

8) The prediction interval with $1-\alpha$ coverage probability and $1-\gamma$ guarantee level is given by the set
\begin{equation}
\left\{y_f\Big|\vert y_f - x_f^T\widehat{\beta}\vert\leq \widehat{c}^*_{1-\alpha + \widehat{d}^*_{1-\gamma}/\sqrt{n}}\right\} .
\label{realPred}
\end{equation}
\label{alg2}
\end{algorithm}

\begin{remark}
In \textit{RBUG}, if we take a conditional expectation on $p^*_{b_1}$ (conditioning on everything except $\varepsilon_{b_2}^*, b_2 = 1,2,...,\mathcal{B}_2$), it becomes
\begin{equation}
\begin{aligned}
\frac{1}{\sqrt{n}}\sum_{i = 1}^n\left(\mathbf{1}_{e^*_i\leq\widehat{c}^*_{1-\alpha}}-\widehat{F}(\widehat{c}^*_{1-\alpha})\right) - \sqrt{n}\left(\widehat{F}\left(\widehat{c}^*_{1-\alpha} + x_f^T(X^TX)^{-1}X^Te^* - \frac{1}{n}\sum_{i=1}^ne^*_i\right) - \widehat{F}(\widehat{c}^*_{1-\alpha})\right)\\
-\frac{1}{\sqrt{n}}\sum_{i = 1}^n\left(\mathbf{1}_{e^*_i<-\widehat{c}^*_{1-\alpha}} - \widehat{F}^-(-\widehat{c}^*_{1-\alpha})\right) + \sqrt{n}\left(\widehat{F}^{-}\left(-\widehat{c}^*_{1-\alpha} + x_f^T(X^TX)^{-1}X^Te^* - \frac{1}{n}\sum_{i=1}^ne^*_i\right) - \widehat{F}^-(-\widehat{c}^*_{1-\alpha})\right)
\label{JJO}
\end{aligned}
\end{equation}
which simulates $-\widetilde{M}_m\left(\frac{\widehat{c}^*_{1-\alpha} + m}{2m}\right) + \widetilde{M}_m^{-}\left(\frac{-\widehat{c}^*_{1-\alpha}+m}{2m}\right)$(for arbitrary chosen integer $m > \vert \widehat{c}^*_{1-\alpha}\vert + 1$)  in the bootstrap world. From the strong law of large numbers, $p^*_{b_1}, b_1 = 1,2,...,\mathcal{B}_1$ can approximate its conditional expectation \eqref{JJO} by letting $\mathcal{B}_2\to\infty$.
The same discussion applies to \textit{PRBUG} as well.
\end{remark}

We focus on proving \textit{RBUG}'s validity, i.e., that prediction interval
(\ref{realPred}) satisfies definition \ref{def1}. We define
\begin{equation}
\begin{aligned}
\widehat{\mathcal{M}}(x) = \sqrt{n}\left(\widehat{F}\left(x+x_f^T(X^TX)^{-1}X^Te^*-\frac{1}{n}\sum_{i=1}^ne^*_i\right) - \frac{1}{n}\sum_{i=1}^n\mathbf{1}_{e^*_i\leq x}\right),\ \
\widehat{\mathcal{S}}(x) = \widehat{\mathcal{M}}(x) -  \widehat{\mathcal{M}}^-(-x)
\end{aligned}
\label{JJYYS}
\end{equation}
and the quantiles
\begin{equation}
\begin{aligned}
c^*_{1-\alpha} = \inf\left\{x\in\mathbf{R}\Big| G^*(x)\geq 1-\alpha\right\},\
d^*_{1-\gamma}(x) = \inf\left\{z\in\mathbf{R}\Big| \mathbf{P}^*\left(-\widehat{S}(x)\leq z\right)\geq 1-\gamma\right\}
\end{aligned}
\label{defD}
\end{equation}
Here $G^*(x) = \mathbf{P}^*\left(\vert y_f^*-x_f^T\widehat{\beta}^*\vert\leq x\right),x\in\mathbf{R}$; $y^*_f, \widehat{\beta}^*$, and $e^*_i, i = 1,2,...,n$ are defined in algorithm \ref{alg2}. For simplicity, we denote $c^*(1-\alpha, 1-\gamma) = c^*_{1-\alpha + d^*_{1-\gamma}(c^*_{1-\alpha})/\sqrt{n}}$. From theorem 1.2.1 of Politis et al. \cite{subsampling}, $\widehat{c}^*_{1-\alpha + \widehat{d}^*_{1-\gamma}/\sqrt{n}}$ converges to $c^*(1-\alpha, 1-\gamma)$ almost surely as $B,\mathcal{B}_1,\mathcal{B}_2\to\infty$. Therefore, the theoretical justification only focuses on $c^*(1-\alpha, 1-\gamma)$.

\begin{theorem}
Consider the \textit{RBUG} algorithm, i.e, algorithm 2 with
 $\widehat{\tau} =  \widehat{\epsilon}$ as  in \eqref{DefResi}.
Suppose assumption 1. to 4. hold true. Then,  for any given $0<\alpha,\gamma<1, \delta>0$,
\begin{equation}
\begin{aligned}
\mathbf{P}\left(\vert\mathbf{P}^*\left(\vert y_f - x^T_f\widehat{\beta}\vert\leq c^*(1-\alpha,1-\gamma)\right) - (1-\alpha)\vert\leq \delta\right)\to 1\\
\mathbf{P}\left(\{\mathbf{P}^*\left(\vert y_f - x^T_f\widehat{\beta}\vert\leq c^*(1-\alpha,1-\gamma)\right)\geq 1-\alpha\}\right)\to 1-\gamma
\end{aligned}
\end{equation}
\label{THMAS}
\end{theorem}

Corollary \ref{coroF} proves the validity of \textit{PRBUG}. We
 define $\mathcal{G}^*(x) = \mathbf{P}^*\left(\vert y_f^*-x_f^T\widehat{\beta}^*\vert\leq x\right),x\in\mathbf{R}$, $C^*_{1-\alpha} = \inf\left\{x\in\mathbf{R}\Big| \mathcal{G}^*(x)\geq 1-\alpha\right\}$, and $D^*_{1-\gamma}(x) = \inf\left\{z\in\mathbf{R}\Big| \mathbf{P}^*\left(-\widehat{S}(x)\leq z\right)\geq 1-\gamma\right\}$. We define $C^*(1-\alpha,1-\gamma) = C^*_{1-\alpha+D^*_{1-\gamma}(C^*_{1-\alpha})/\sqrt{n}}$.

\begin{corollary}
Consider the \textit{PRBUG} algorithm, i.e, algorithm 2 with $\widehat{\tau} = \widehat{r}$.
Suppose assumptions 1. to 4. hold true. Then,  for any given $0<\alpha,\gamma<1, \delta > 0$,
\begin{equation}
\begin{aligned}
\mathbf{P}\left(\vert\mathbf{P}^*\left(\vert y_f - x_f^T\widehat{\beta}\vert\leq C^*(1-\alpha,1-\gamma)\right) - (1-\alpha)\vert\leq \delta\right)\to 1\\
\mathbf{P}\left(\{\mathbf{P}^*\left(\vert y_f - x^T_f\widehat{\beta}\vert\leq C^*(1-\alpha,1-\gamma)\right)\geq 1-\alpha\}\right)\to 1-\gamma .
\end{aligned}
\label{PrdTheo}
\end{equation}
\label{coroF}
\end{corollary}

\section{Numerical justification}
\label{Numerics}
This section applies numerical simulations to demonstrate the finite sample performance of \textit{RBUG/PRBUG}. The alternatives are residual-based bootstrap and \textit{MF/MB} bootstrap. Figure \ref{Toy} plots point-wise prediction intervals for the linear model $y_i = 0.8 + 0.5x_i + \epsilon_i, i = 1,2,...,100$. I.i.d. residuals are generated by normal distribution with mean $0$ and variance $1$.
\begin{figure}[htbp]
  \centering
  \includegraphics[width = 3in]{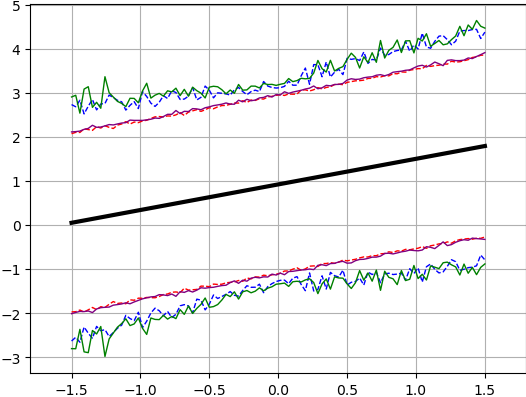}\\
  \caption{Point-wise prediction intervals for the linear model $y_i = 0.8 + 0.5x_i + \epsilon_i, i = 1,2,...,100$. Black line, red dashed lines, solid purple lines, blue dashed lines, solid green lines respectively plot predictors, and point-wise prediction intervals generated by residual-based bootstrap, \textit{MF/MB bootstrap}, \textit{RBUG} and \textit{PRBUG}. The nominal coverage probability is $95\%$, and the nominal guarantee level is $90\%$.}
  \label{Toy}
\end{figure}

Our  linear model of choice is denoted as the   \textbf{Experiment model}
and defined as follows: $y = X\beta + \epsilon$, and $\beta$'s dimension is $15$. $\beta = (\beta_0, \beta_1,...,\beta_{14})^T$, $\beta_0 = 1$, $\beta_1 = 0.5$, $\beta_2 = -1.0$, $\beta_3 = -0.5$, and $\beta_i = 0$ for $i > 3$. The design matrix $X$ is generated by i.i.d.
standard normal random variables, and is fixed in each experiment. The new regressor $x_f = (x_{f,0},...,x_{f,14})^T$ is given by $x_{f,i} = 0.1 \times i, i = 0,1,...,14$. The i.i.d. error vector  $\epsilon$ is generated by various distributions.
We choose the sample size $n = 50, 100, 400, 800, 1600$. The result is demonstrated in table \ref{tableMain} and figure \ref{Hist}. When the sample size is very small, the \textit{MF/MB bootstrap} alleviates the residual-based bootstrap's under-coverage nature. Yet this modification does not change the asymptotic guarantee level. On the other hand, the \textit{RBUG} and the \textit{PRBUG} algorithms improve the residual-based bootstrap's performance by controlling the asymptotic guarantee level.

\begin{table}[htbp]
  \centering
  \caption{Performance of different algorithms on the \textbf{Experiment model}. The nominal coverage probability is $95\%$, and the nominal guarantee level is $85\%$. In the `Residual type' column, `Normal' means standard normal distribution, and `Laplace' means Laplace distribution with location $0$ and scale $1/\sqrt{2}$; this setting  ensures the errors have variance $1$. In the `Algorithm' column, `RB' means residual-based bootstrap, and `\textit{MF/MB}' means \textit{MF/MB} bootstrap. We choose $B,\mathcal{B}_1,\mathcal{B}_2 $ all equal to
$ 2500$ in \textit{RBUG} and \textit{PRBUG}.}
  \label{tableMain}
  \scriptsize
  \begin{tabular}{l l l l l l l l}
  \hline\hline\hline
  Residual type & Sample size & Algorithm       & \multicolumn{4}{l}{Quantiles of coverage probabilities}   & Guarantee level\\
                &             &                 & $25\%$    & $45\%$      & $65\%$        & $85\%$          &              \\
  Normal        &   50        & RB              & $85.8\%$  & $89.6\%$    & $92.4\%$      & $94.9\%$        & $14.3\%$     \\
                &             & \textit{MF/MB}  & $97.1\%$  & $98.2\%$    & $98.9\%$      & $99.5\%$        & $87.1\%$     \\
                &             & \textit{RBUG}   & $98.2\%$  & $99.2\%$    & $99.6\%$      & $99.9\%$        & $93.0\%$     \\
                &             & \textit{PRBUG}  & $99.9\%$  & $100.0\%$   & $100.0\%$     & $100.0\%$       & $99.9\%$     \\
                \hline
  Normal        &   100       & RB              & $90.9\%$  & $92.5\%$    & $93.9\%$      & $95.4\%$        & $19.3\%$     \\
                &             & \textit{MF/MB}  & $95.4\%$  & $96.5\%$    & $97.4\%$      & $98.2\%$        & $80.2\%$     \\
                &             & \textit{RBUG}   & $96.7\%$  & $97.8\%$    & $98.7\%$      & $99.5\%$        & $90.9\%$     \\
                &             & \textit{PRBUG}  & $98.8\%$  & $99.4\%$    & $99.7\%$      & $99.9\%$        & $99.7\%$     \\
                \hline
  Normal        &   400       & RB              & $93.6\%$  & $94.3\%$    & $95.0\%$      & $95.8\%$        & $34.4\%$     \\
                &             & \textit{MF/MB}  & $94.6\%$  & $95.3\%$    & $95.9\%$      & $96.5\%$        & $64.5\%$     \\
                &             & \textit{RBUG}   & $95.6\%$  & $96.2\%$    & $96.7\%$      & $97.4\%$        & $87.9\%$     \\
                &             & \textit{PRBUG}  & $96.4\%$  & $96.9\%$    & $97.3\%$      & $97.9\%$        & $96.0\%$     \\
                \hline
  Normal        &   800       & RB              & $94.2\%$  & $94.7\%$    & $95.1\%$      & $95.7\%$        & $42.4\%$     \\
                &             & \textit{MF/MB}  & $94.6\%$  & $95.1\%$    & $95.6\%$      & $96.1\%$        & $60.8\%$     \\
                &             & \textit{RBUG}   & $95.3\%$  & $95.8\%$    & $96.3\%$      & $96.8\%$        & $85.6\%$     \\
                &             & \textit{PRBUG}  & $95.8\%$  & $96.2\%$    & $96.6\%$      & $97.0\%$        & $94.4\%$     \\
                \hline
  Normal        &  1600       & RB              & $94.4\%$  & $94.8\%$    & $95.2\%$      & $95.2\%$        & $44.3\%$     \\
                &             & \textit{MF/MB}  & $94.6\%$  & $95.0\%$    & $95.4\%$      & $95.8\%$        & $55.3\%$     \\
                &             & \textit{RBUG}   & $95.2\%$  & $95.6\%$    & $96.0\%$      & $96.4\%$        & $84.4\%$     \\
                &             & \textit{PRBUG}  & $95.4\%$  & $95.8\%$    & $96.1\%$      & $96.5\%$        & $90.2\%$     \\
  \hline\hline
  Laplace       & 50          & RB              & $87.4\%$  & $90.7\%$    & $92.7\%$      & $94.9\%$        & $13.6\%$     \\
                &             & \textit{MF/MB}  & $96.0\%$  & $97.4\%$    & $98.2\%$      & $98.8\%$        & $83.5\%$     \\
                &             & \textit{RBUG}   & $97.1\%$  & $98.2\%$    & $98.9\%$      & $99.5\%$        & $89.9\%$     \\
                &             & \textit{PRBUG}  & $99.4\%$  & $99.7\%$    & $99.8\%$      & $99.9\%$        & $99.9\%$     \\
                \hline
  Laplace       & 100         & RB              & $91.0\%$  & $92.8\%$    & $94.0\%$      & $95.3\%$        & $20.3\%$     \\
                &             & \textit{MF/MB}  & $94.4\%$  & $95.6\%$    & $96.5\%$      & $97.5\%$        & $67.3\%$     \\
                &             & \textit{RBUG}   & $95.7\%$  & $97.0\%$    & $97.9\%$      & $98.7\%$        & $83.7\%$     \\
                &             & \textit{PRBUG}  & $97.6\%$  & $98.4\%$    & $98.9\%$      & $99.5\%$        & $97.3\%$     \\
                \hline
  Laplace       &400          & RB              & $93.6\%$  & $94.4\%$    & $95.0\%$      & $95.8\%$        & $36.3\%$     \\
                &             & \textit{MF/MB}  & $94.3\%$  & $95.0\%$    & $95.6\%$      & $96.3\%$        & $55.1\%$     \\
                &             & \textit{RBUG}   & $95.3\%$  & $95.9\%$    & $96.5\%$      & $97.1\%$        & $82.3\%$     \\
                &             & \textit{PRBUG}  & $95.8\%$  & $96.4\%$    & $96.9\%$      & $97.4\%$        & $93.4\%$     \\
                \hline
  Laplace       &800          & RB              & $94.1\%$  & $94.7\%$    & $95.1\%$      & $95.7\%$        & $41.0\%$     \\
                &             & \textit{MF/MB}  & $94.5\%$  & $94.9\%$    & $95.4\%$      & $96.0\%$        & $53.6\%$     \\
                &             & \textit{RBUG}   & $95.3\%$  & $95.7\%$    & $96.1\%$      & $96.7\%$        & $83.1\%$     \\
                &             & \textit{PRBUG}  & $95.5\%$  & $96.0\%$    & $96.4\%$      & $96.8\%$        & $91.6\%$     \\
                \hline
  Laplace       &1600         & RB              & $94.4\%$  & $94.8\%$    & $95.2\%$      & $95.7\%$        & $47.5\%$     \\
                &             & \textit{MF/MB}  & $94.6\%$  & $95.0\%$    & $95.4\%$      & $95.8\%$        & $54.6\%$     \\
                &             & \textit{RBUG}   & $95.2\%$  & $95.6\%$    & $95.9\%$      & $96.4\%$        & $82.4\%$     \\
                &             & \textit{PRBUG}  & $95.4\%$  & $95.7\%$    & $96.1\%$      & $96.5\%$        & $86.0\%$     \\
  \hline\hline\hline
  \end{tabular}
\end{table}

\begin{figure}[htbp]
    \subfigure[Normal, Residual-based bootstrap]{
        \includegraphics[width=1.7in]{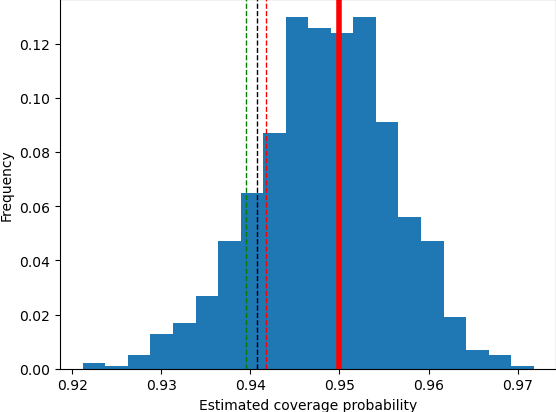}
    }
    \subfigure[Normal, \textit{MF/MB} bootstrap]{
        \includegraphics[width=1.7in]{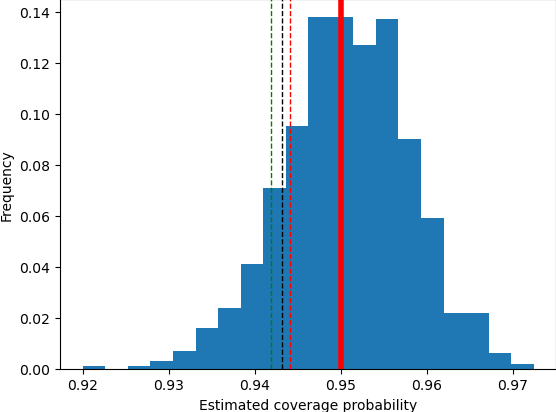}
    }
    \subfigure[Normal, \textit{RBUG}]{
        \includegraphics[width=1.7in]{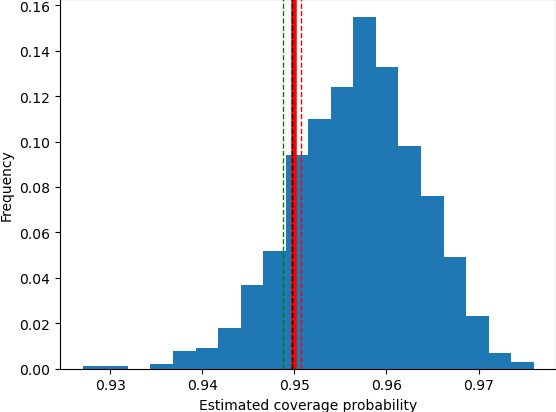}
    }
    \subfigure[Normal, \textit{PRBUG}]{
        \includegraphics[width=1.7in]{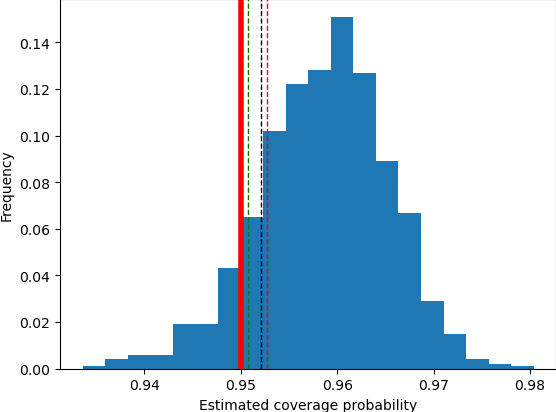}
    }

    \subfigure[Laplace, Residual-based bootstrap ]{
        \includegraphics[width=1.7in]{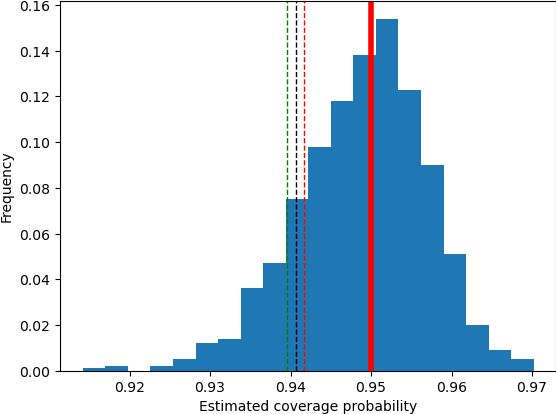}
    }
    \subfigure[Laplace, \textit{MF/MB} bootstrap]{
        \includegraphics[width=1.7in]{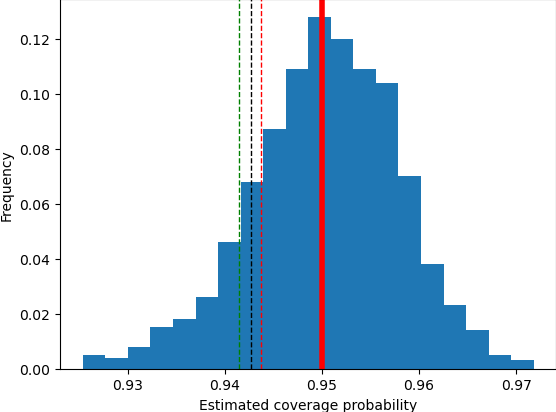}
    }
    \subfigure[Laplace, \textit{RBUG}]{
        \includegraphics[width=1.7in]{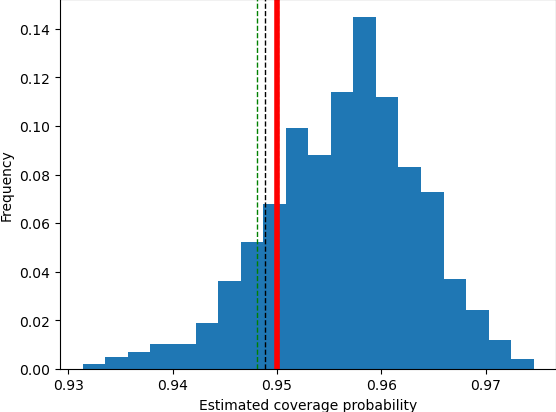}
    }
    \subfigure[Laplace, \textit{PRBUG}]{
        \includegraphics[width=1.7in]{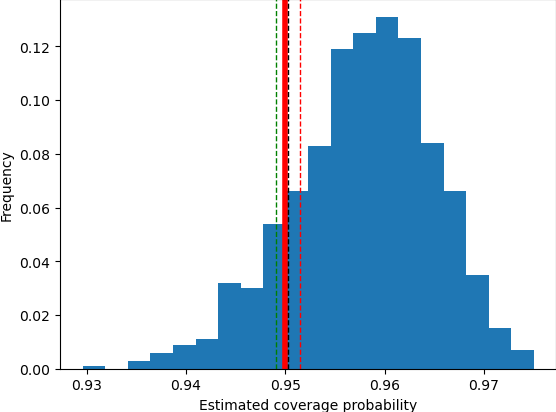}
    }
  \caption{Histograms for the conditional coverage probabilities. We use the \textbf{Experiment model} with sample size $1600$ and different error distributions--see table \ref{tableMain}.
The solid red line indicates the nominal coverage probability($95\%$); the green dashed line, the black dashed line, and the red dashed line respectively indicates the $12\%, 15\%, 18\%$ quantile of conditional coverage probabilities. The nominal guarantee level for \textit{RBUG} and \textit{PRBUG} is $85\%$.}
  \label{Hist}
\end{figure}

\pagebreak
\section{Conclusion}
Focusing on the fixed design linear model, in this paper we derive the asymptotic distribution of the difference between the conditional coverage probability of a
nominal prediction interval $\mathbf{P}^*\left(\vert y_f - x_f^T\widehat{\beta}\vert\leq x\right)$ and the conditional coverage probability of a prediction interval for residual-based bootstrapped observations $\mathbf{P}^*\left(\vert y_f^* - x_f^T\widehat{\beta}^*\vert\leq x\right)$. According to this result, the prediction interval generated by residual-based bootstrap has
approximately $50\%$ probability to yield conditional under-coverage.

We then develop a new bootstrap algorithm that generates prediction intervals with arbitrarily assigned conditional coverage probability and guarantee level, and prove its asymptotic validity. Our theoretical results are corroborated by several finite-sample
simulations.

 Residual-based and the MF/MB  bootstrap are  widely used for prediction in numerous settings like nonparametric/nonlinear regression, quantile regression,  time series
analysis (regression with dependent errors, autoregression, etc.), and others.
 We expect our ideas to be applicable   in those settings as well;
future work will  address the details.  Furthermore, the case of
 high-dimensional linear regression is of  current interest, i.e., where $p$ is allowed to diverge as $n\to \infty$; this can also be the subject of future work.

\label{ch5}

\section*{Acknowledgement}
We appreciate the valuable suggestions from Dr. Xiaoou Pan that helped improve the quality of this paper. This research was partially supported by NSF Grant DMS 19-14556.

\bibliographystyle{unsrt}
\bibliography{Draft}

\begin{thebibliography}{10}

\bibitem{doi:10.1080/01621459.1996.10476934}
Jun Shao.
\newblock Bootstrap model selection.
\newblock {\em Journal of the American Statistical Association},
  91(434):655--665, 1996.

\bibitem{xie2009}
Huiliang Xie and Jian Huang.
\newblock Scad-penalized regression in high-dimensional partially linear
  models.
\newblock {\em Ann. Statist.}, 37(2):673--696, 04 2009.

\bibitem{liu2013}
Hanzhong Liu and Bin Yu.
\newblock Asymptotic properties of lasso+mls and lasso+ridge in sparse
  high-dimensional linear regression.
\newblock {\em Electron. J. Statist.}, 7:3124--3169, 2013.

\bibitem{freedman1981}
David~A. Freedman.
\newblock Bootstrapping regression models.
\newblock {\em Ann. Statist.}, 9(6):1218--1228, 11 1981.

\bibitem{PrdInference}
Seymour Geisser.
\newblock {\em Predictive Inference: An Introduction}.
\newblock Chapman and Hall/CRC, 1 edition, 1993.

\bibitem{Model-based}
Dimitris~N. Politis.
\newblock {\em Model-Free Prediction and Regression}.
\newblock Springer International Publishing, USA, 1 edition, 2015.

\bibitem{ESl}
Trevor Hastie, Robert Tibshirani, and Jerome Friedman.
\newblock {\em The Elements of Statistical Learning}.
\newblock Springer-Verlag New York, 2 edition, 2009.

\bibitem{NEURIPS2019_5103c358}
Yaniv Romano, Evan Patterson, and Emmanuel Candes.
\newblock Conformalized quantile regression.
\newblock In {\em Advances in Neural Information Processing Systems},
  volume~32, pages 3543--3553. Curran Associates, Inc., 2019.

\bibitem{doi:10.1080/01621459.1985.10478220}
Robert~A. Stine.
\newblock Bootstrap prediction intervals for regression.
\newblock {\em Journal of the American Statistical Association},
  80(392):1026--1031, 1985.

\bibitem{ConPrd}
Glenn Shafer and Vladimir Vovk.
\newblock A tutorial on conformal prediction.
\newblock {\em Journal of Machine Learning Research}, 9:371--421, 2008.

\bibitem{AlgorithmWord}
Vladimir Vovk, Alexander Gammerman, and Glenn Shafer.
\newblock {\em Algorithmic Learning in a Random World}.
\newblock Springer, Boston, MA, 1 edition, 2005.

\bibitem{doi:10.1111/rssb.12021}
Jing Lei and Larry Wasserman.
\newblock Distribution-free prediction bands for non-parametric regression.
\newblock {\em Journal of the Royal Statistical Society: Series B (Statistical
  Methodology)}, 76(1):71--96, 2014.

\bibitem{doi:10.1080/01621459.2017.1307116}
Jing Lei, Max G'Sell, Alessandro Rinaldo, Ryan~J. Tibshirani, and Larry
  Wasserman.
\newblock Distribution-free predictive inference for regression.
\newblock {\em Journal of the American Statistical Association},
  113(523):1094--1111, 2018.

\bibitem{doi:10.1002/sta4.261}
Matteo Sesia and Emmanuel~J. Cand\`es.
\newblock A comparison of some conformal quantile regression methods.
\newblock {\em Stat}, 9(1):e261, 2020.
\newblock e261 sta4.261.

\bibitem{Stochastic}
Erhan \c{C}inlar.
\newblock {\em Probability and Stochastics}.
\newblock Springer-Verlag New York, 1 edition, 2011.

\bibitem{Rpackage}
Ryan Tibshirani, Nina Jeliazkova, and Erin LeDell.
\newblock Conformal inference r project.
\newblock \url{https://github.com/ryantibs/conformal}.
\newblock Accessed: 2020-12-19.

\bibitem{billing}
Patrick Billingsley.
\newblock {\em Convergence of probability measures}.
\newblock Wiley Series in Probability and Statistics: Probability and
  Statistics. John Wiley \& Sons Inc., New York, second edition, 1999.
\newblock A Wiley-Interscience Publication.

\bibitem{subsampling}
Dimitris~N. Politis, Joseph~P. Romano, and Michael Wolf.
\newblock {\em Subsampling}.
\newblock Springer, New York, NY, USA, 1st edition, 1999.

\bibitem{doi:10.1080/01621459.1990.10476225}
Lori~A. Thombs and William~R. Schucany.
\newblock Bootstrap prediction intervals for autoregression.
\newblock {\em Journal of the American Statistical Association},
  85(410):486--492, 1990.

\bibitem{PAN2016467}
Li~Pan and Dimitris~N. Politis.
\newblock Bootstrap prediction intervals for markov processes.
\newblock {\em Computational Statistics $\&$ Data Analysis}, 100:467 -- 494,
  2016.

\bibitem{PAN20161}
Li~Pan and Dimitris~N. Politis.
\newblock Bootstrap prediction intervals for linear, nonlinear and
  nonparametric autoregressions.
\newblock {\em Journal of Statistical Planning and Inference}, 177:1 -- 27,
  2016.

\bibitem{10.2307/2240410}
Peter~J. Bickel and David~A. Freedman.
\newblock Some asymptotic theory for the bootstrap.
\newblock {\em The Annals of Statistics}, 9(6):1196--1217, 1981.

\bibitem{10.2307/2242984}
Marjorie~G. Hahn.
\newblock Conditions for sample-continuity and the central limit theorem.
\newblock {\em The Annals of Probability}, 5(3):351--360, 1977.

\bibitem{limitEmpirica;}
Jean Jacod and Albert~N Shiryaev.
\newblock {\em Limit Theorems for Stochastic Processes}.
\newblock Springer-Verlag Berlin Heidelberg, 2003.

\bibitem{mathStat}
Jun Shao.
\newblock {\em Mathematical Statistics}.
\newblock Springer-Verlag New York, 2003.

\bibitem{10.1093/biomet/asz020}
Mengyu Xu, Danna Zhang, and Wei~Biao Wu.
\newblock {Pearson's chi-squared statistics: approximation theory and beyond}.
\newblock {\em Biometrika}, 106(3):716--723, 04 2019.

\bibitem{10.2307/2237541}
R.~Ranga Rao.
\newblock Relations between weak and uniform convergence of measures with
  applications.
\newblock {\em The Annals of Mathematical Statistics}, 33(2):659--680, 1962.

\bibitem{empirical_process}
Hira~L. Koul.
\newblock {\em Weighted Empirical Processes in Dynamic Nonlinear Models}.
\newblock Springer-Verlag New York, 2002.

\bibitem{10.5555/2422911}
Roger~A. Horn and Charles~R. Johnson.
\newblock {\em Matrix Analysis}.
\newblock Cambridge University Press, USA, 2nd edition, 2012.

\bibitem{optimalTransport}
C\'edric Villani.
\newblock {\em Optimal Transport}.
\newblock Springer, Berlin, Heidelberg, 1st edition, 2009.

\bibitem{Linreg}
George~A.F. Seber and Alan~J. Lee.
\newblock {\em Linear Regression Analysis}.
\newblock John Wiley $\&$ Sons, 2 edition, 2003.

\end{thebibliography}

\clearpage
\appendix

\section{Proof of theorem \ref{theoPd}}
\label{lemmaS}
Suppose random variables $A,B$ satisfy $\vert A - B\vert\leq \delta,\delta>0$, then $\forall x\in\mathbf{R}$, $-\mathbf{1}_{x-\delta<B\leq x}\leq\mathbf{1}_{A\leq x} - \mathbf{1}_{B\leq x}\leq \mathbf{1}_{x<B\leq x+\delta}$, which implies
\begin{equation}
\mathbf{E}\vert\mathbf{1}_{A\leq x} - \mathbf{1}_{B\leq x}\vert\leq \mathbf{E}\vert\mathbf{1}_{A\leq x} - \mathbf{1}_{B\leq x}\vert\mathbf{1}_{\vert A -B\vert\leq\delta} + Prob(\vert A - B\vert>\delta)\leq Prob(\vert A - B\vert>\delta) + Prob(x-\delta<B\leq x+\delta)
\label{Deltas}
\end{equation}
For $\forall t_i\in[0,1], s_i\in\mathbf{R}, i = 1,2,...,r$, we define $\widetilde{\mathcal{M}}_m$ as in \eqref{AsymptoticM} and $t^{'}_i = 2mt_i - m$,
\begin{equation}
\begin{aligned}
0\leq \lim_{n\to\infty}\mathbf{E}\left(\sum_{i = 1}^rs_i\widetilde{\mathcal{M}}_m(t_i)\right)^2
= \lim_{n\to\infty}\sum_{i = 1}^r\sum_{j = 1}^r s_is_j\sigma^2F^{'}(t^{'}_i)F^{'}(t^{'}_j)\left(x_f^T\left(\frac{X^TX}{n}\right)^{-1}x_f + 1 - 2x_f^T\left(\frac{X^TX}{n}\right)^{-1}\overline{x}_n\right)\\
-\lim_{n\to\infty}\sum_{i=1}^r\sum_{j=1}^r s_is_j\left(x_f^T\left(\frac{X^TX}{n}\right)^{-1}\overline{x}_n - 1\right)\times \left(F^{'}(t_i^{'})H(t_j^{'}) + F^{'}(t_j^{'})H(t^{'}_i))\right)\\
+\lim_{n\to\infty}\sum_{i = 1}^r\sum_{j = 1}^r s_is_j\left(F(\min(t^{'}_i,t^{'}_j))- F(t^{'}_i)F(t^{'}_j)\right)
=\sum_{i = 1}^r\sum_{j = 1}^r s_is_j\mathcal{V}(t^{'}_i,t^{'}_j)
\end{aligned}
\label{VPositive}
\end{equation}
Therefore, $\forall t_i\in\mathbf{R}, i = 1,2,...,r$, the matrix $\{\mathcal{V}(t_i,t_j)\}_{i,j = 1,2,...,r}$ is symmetric positive semi-definite.
Suppose assumption 3),
\begin{equation}
\mathbf{E}\widehat{\sigma}^2\leq \frac{2}{n}\sum_{i = 1}^n \mathbf{E}\left(\epsilon_i - \frac{\sum_{j = 1}^n\epsilon_j}{n}\right)^2 + \frac{2}{n}\sum_{i = 1}^n \mathbf{E}((x_i - \overline{x}_n)^T(X^TX)^{-1}X^T\epsilon)^2\leq 2\sigma^2 + \frac{8M^2\sigma^2}{n}\Vert\left(\frac{X^TX}{n}\right)^{-1}\Vert_2
\label{SigmaBound}
\end{equation}
Here $\widehat{\sigma}$ is defined in \eqref{sigmaHat}, and $\Vert \left(\frac{X^TX}{n}\right)^{-1}\Vert_2$ is the matrix 2-norm of $\left(\frac{X^TX}{n}\right)^{-1}$. These three results are frequently used in the proofs. This section then introduces some lemmas.

\begin{lemma}
Suppose assumption 1. to 4. hold true.

1. $\forall 0<m\in\mathbf{N}$, $\exists$ a Gaussian process $\mathcal{M}_m$ in $\mathbf{D}$ satisfying \eqref{stopro} and having continuous sample paths almost surely.

2. For any given $\xi>0$,
\begin{equation}
\lim_{\delta\to 0,\delta>0} \mathbf{P}\left(\sup_{y,z\in[0,1],\vert y - z\vert<\delta}\vert\mathcal{M}_m(y) - \mathcal{M}_m(z)\vert > \xi\right) = 0
\label{ContinuA}
\end{equation}
and if a sequence of stochastic processes $\widetilde{\mathcal{M}}_{m,n}, n = 1,2,...$ satisfy $\widetilde{\mathcal{M}}_{m,n}\to_{\mathcal{L}} \mathcal{M}_m$ under Skohord topology, $\exists \delta>0$ such that
\begin{equation}
\limsup_{n\to\infty} \mathbf{P}\left(\sup_{y,z\in[0,1],\vert y - z\vert<\delta}\vert\widetilde{\mathcal{M}}_{m,n}(y) - \widetilde{\mathcal{M}}_{m,n}(z)\vert \geq \xi\right)\leq \xi
\label{approxConverge}
\end{equation}
\label{lemmaExt}
\end{lemma}

\begin{proof}[proof of lemma \ref{lemmaExt}]
From \eqref{VPositive}, for any $t_i\in[0,1], i = 1,2,...,r$, the random vector $(\mathcal{M}_m(t_1),...,\mathcal{M}_m(t_r))^T$ has joint normal distribution with mean $0$ and covariance matrix $\{\mathcal{V}(t_i,t_j)\}_{i,j = 1,2,...,r}$, so the consistency conditions in Kolmogorov extension theorem are satisfied. $\forall 0\leq t_1\leq t\leq t_2\leq 1$,
\begin{equation}
\begin{aligned}
\mathbf{E}\vert\mathcal{M}_m(t)-\mathcal{M}_m(t_1)\vert^2\vert \mathcal{M}_m(t)-\mathcal{M}_m(t_2)\vert^2\leq \frac{1}{2}\left(\mathbf{E}\vert\mathcal{M}_m(t)-\mathcal{M}_m(t_1)\vert^4+\mathbf{E}\vert \mathcal{M}_m(t)-\mathcal{M}_m(t_2)\vert^4\right)\\
\leq \frac{3}{2}\left(\mathbf{E}(\mathcal{M}_m(t)-\mathcal{M}_m(t_1))^2+\mathbf{E}(\mathcal{M}_m(t)-\mathcal{M}_m(t_2))^2\right)^2
\end{aligned}
\label{fourMoment}
\end{equation}
We define $t^{'} = 2mt - m$. Form assumption 1), $\exists$ a constant  $C > 0$ with
\begin{equation}
\begin{aligned}
\mathbf{E}(\mathcal{M}_m(t)-\mathcal{M}_m(t_1))^2 = \sigma^2(x_f^TA^{-1}x_f+1-2x_f^TA^{-1}b)(F^{'}(t^{'})-F^{'}(t^{'}_1))^2
+F(t^{'})-F(t^{'}_1) - (F(t)-F(t^{'}_1))^2 \\
- 2(x_f^TA^{-1}b-1)(F^{'}(t^{'})-F^{'}(t^{'}_1))(H(t^{'})-H(t^{'}_1))\leq C(t - t_1)\\
\end{aligned}
\label{inter}
\end{equation}
Similarly, $\mathbf{E}(\mathcal{M}_m(t_2)-\mathcal{M}_m(t))^2\leq C(t_2 - t)$, and \eqref{fourMoment} implies $\mathbf{E}\vert\mathcal{M}_m(t)-\mathcal{M}_m(t_1)\vert^2\vert \mathcal{M}_m(t)-\mathcal{M}_m(t_2)\vert^2\leq \frac{3}{2}C^2(t_2 - t_1)^2$. Therefore, (13.15) in \cite{billing} is satisfied by setting $\alpha = \beta = 1$, and choosing the non-decreasing, continuous function $F(x) = \frac{\sqrt{3}}{\sqrt{2}}Cx$. \eqref{inter} also implies (13.16) in \cite{billing}. From theorem 13.6 in \cite{billing}, $\exists \mathcal{M}_m\in\mathbf{D}$ satisfying \eqref{stopro}. According to \eqref{inter}, $\mathbf{E}(\mathcal{M}_m(t)-\mathcal{M}_m(t_1))^4\leq 3C^2(t-t_1)^2$, so theorem 2.3 in \cite{10.2307/2242984} is satisfied by choosing $r = 4$ and the function $f(x) = 3C^2x^2\Rightarrow \int_{[0,1]}x^{-(r+1)/r}f^{1/r}(x)dx = 4(3C^2)^{1/4}<\infty$. In particular, we can choose $\mathcal{M}_m\in\mathbf{D}$ such that $\vert\mathcal{M}_m(t) - \mathcal{M}_m(t_1)\vert\leq AH(\vert t - t_1\vert)$ almost surely, A is a random variable with $\mathbf{E}A^4<\infty$, $H$ is a continuous nondecreasing function on $[0,1]$, $H(0) = 0$. This implies $\mathcal{M}_m$ has continuous sample paths almost surely.

We prove \eqref{ContinuA} by
\begin{equation}
\begin{aligned}
\mathbf{P}\left(\sup_{y,z\in[0,1],\vert y - z\vert<\delta}\vert\mathcal{M}_m(y) - \mathcal{M}_m(z)\vert > \xi\right)\leq \frac{\mathbf{E}A^4}{\xi^4}\times H^4(\delta)
\end{aligned}
\end{equation}
.

For any given $\delta>0$, we define a function $h_\delta(f) = \sup_{x,y\in[0,1],\vert x - y\vert < \delta}\vert f(x) - f(y)\vert, f\in\mathbf{D}$. From section 12, \cite{billing}, if $f_n, n = 1,...$ converges to $f$ in $\mathbf{D}$, then $\exists$ strictly increasing mappings $\lambda_n:[0,1]\to[0,1], n = 1,2,...$ such that $\lim_{n\to\infty}\sup_{x\in[0,1]}\vert\lambda_n(x) - x\vert = 0$, and $\lim_{n\to\infty}\sup_{x\in[0,1]}\vert f_n(\lambda_n(x)) - f(x)\vert = 0$; so
\begin{equation}
\begin{aligned}
\vert h_\delta(f_n) - h_\delta(f)\vert\leq \sup_{x,y\in[0,1],\vert x - y\vert<\delta}\vert f_n(x) - f_n(y) - f(x) + f(y)\vert
\leq \sup_{x\in[0,1]}\vert f_n(x) - f(x)\vert + \sup_{y\in[0,1]}\vert f_n(y) - f(y)\vert\\
\leq 2(\sup_{x\in[0,1]}\vert f_n(x) - f(\lambda_n^{-1}(x))\vert + \sup_{x\in[0,1]}\vert f(\lambda_n^{-1}(x)) - f(x)\vert)
\end{aligned}
\label{Hdels}
\end{equation}
If $f$ is continuous on $[0,1]$, then $\lim_{n\to\infty} \vert h_\delta(f_n) - h_\delta(f)\vert = 0$. For $\mathcal{M}_m$ is continuous almost surely, and $\mathbf{R},\mathbf{D}$ are Polish spaces(theorem 12.2 in \cite{billing}), 3.8, page 348 in \cite{limitEmpirica;} implies $h_\delta(\widetilde{\mathcal{M}}_{m,n})\to_{\mathcal{L}} h_\delta(\mathcal{M}_m)$, and theorem 1.9, \cite{mathStat} implies
\begin{equation}
\begin{aligned}
\limsup_{n\to\infty} \mathbf{P}\left(\sup_{x,y\in[0,1],\vert x - y\vert<\delta}\vert\widetilde{\mathcal{M}}_{m,n}(x) - \widetilde{\mathcal{M}}_{m,n}(y)\vert \geq \xi\right)\leq \mathbf{P}\left(\sup_{x,y\in[0,1],\vert x - y\vert < \delta}\vert\mathcal{M}_m(x) - \mathcal{M}_m(y)\vert\geq \xi\right) < \xi
\end{aligned}
\label{limcont}
\end{equation}
for sufficiently small $\delta>0$.
\end{proof}

$\widetilde{\mathcal{M}}_{m,n}$ may have discontinuities. However, the discontinuities vanish asymptotically according to lemma \ref{lemmaExt}. Combine lemma \ref{lemmaExt} with \eqref{Deltas}, we derive the following corollary:

\begin{corollary}
Suppose assumption 1. to 4. hold true. Then for any given $0<c<1/4$,
\begin{equation}
\lim_{\delta \to 0}\sup_{\vert x - y\vert+\vert z-w\vert<\delta}\vert\mathbf{P}(\mathcal{M}_m(x) - \mathcal{M}_m(1-x)\leq z) - \mathbf{P}(\mathcal{M}_m(y)-\mathcal{M}_m(1-y)\leq w)\vert = 0
\label{firProb}
\end{equation}
and if $\widetilde{\mathcal{M}}_{m,n}\to_{\mathcal{L}}\mathcal{M}_m$, then
\begin{equation}
\lim_{n\to\infty}\sup_{x\in[\frac{1}{2}+c,1-c],z\in\mathbf{R}}\vert\mathbf{P}(\widetilde{\mathcal{M}}_{m,n}(x) - \widetilde{\mathcal{M}}_{m,n}^-(1-x)\leq z) - \mathbf{P}(\mathcal{M}_m(x) - \mathcal{M}_m(1-x)\leq z)\vert = 0
\label{SecRes}
\end{equation}
Here $(x,z),(y,w)\in[\frac{1}{2} + c, 1 - c]\times \mathbf{R}$.
\label{coroLink}
\end{corollary}

\begin{proof}[proof of corollary \ref{coroLink}]
Without loss of generality, we assume $z\leq w$. From \eqref{Deltas}, $\forall \xi>0$,
\begin{equation}
\begin{aligned}
\vert\mathbf{P}(\mathcal{M}_m(x) - \mathcal{M}_m(1-x)\leq z) - \mathbf{P}(\mathcal{M}_m(y)-\mathcal{M}_m(1-y)\leq w)\vert\leq \mathbf{P}(\vert\mathcal{M}_m(x) - \mathcal{M}_m(y)\vert>\xi/2)\\
+\mathbf{P}(\vert\mathcal{M}_m(1-x) -\mathcal{M}_m(1-y)\vert>\xi/2)
 +\mathbf{P}(z-\xi<\mathcal{M}_m(y)-\mathcal{M}_m(1-y)\leq z+\xi)+\mathbf{P}(z<\mathcal{M}_m(y)-\mathcal{M}_m(1-y)\leq w)
\end{aligned}
\label{probCur}
\end{equation}
We define $y^{'} = 2my - m$. From assumption 4., $\min_{y\in[\frac{1}{2}+c, 1-c]}\mathcal{U}(y^{'})>0$, so   $\mathbf{P}(z<\mathcal{M}_m(y)-\mathcal{M}_m(1-y)\leq w) = \Phi\left(\frac{w}{\sqrt{\mathcal{U}(y^{'})}}\right) - \Phi\left(\frac{z}{\sqrt{\mathcal{U}(y^{'})}}\right)\leq \frac{\delta}{\min_{y\in[\frac{1}{2}+c, 1-c]}\sqrt{\mathcal{U}(y^{'})}}$. Similarly $\mathbf{P}(z-\xi<\mathcal{M}_m(y)-\mathcal{M}_m(1-y)\leq z+\xi)\leq \frac{2\xi}{\min_{y\in[\frac{1}{2}+c, 1-c]}\sqrt{\mathcal{U}(y^{'})}}$. \eqref{firProb} is proved by applying lemma \ref{lemmaExt} to \eqref{probCur}.

$\forall x\in[\frac{1}{2}+c, 1-c]$, define $g_x:\mathbf{D}\to\mathbf{R}:g_x(f) = f(x) - f^-(1-x)$. We use the same notation as \eqref{Hdels}. If $f_n$ converges to $f$ in $\mathbf{D}$ and $f$ is continuous, $\vert g_x(f_n) - g_x(f)\vert\leq \vert f_n(x) - f(\lambda^{-1}_n(x))\vert + \vert f(\lambda^{-1}_n(x)) - f(x)\vert + \limsup_{t\to1-x,t<1-x}\vert f_n(t) - f(\lambda^{-1}_n(t))\vert + \limsup_{t\to1-x,t<1-x}\vert f(\lambda^{-1}_n(t)) - f(t)\vert$, which tends to $0$ as $n\to\infty$. 3.8, page 348 in \cite{limitEmpirica;} implies $g_x(\widetilde{\mathcal{M}}_{m,n})\to_{\mathcal{L}} g_x(\mathcal{M}_m)$. $\forall \psi>0, t\in\mathbf{R}$, we define $G_0(x) = (1-\min(1,\max(x,0))^4)^4$, and $G_{\psi, t}(x) = G_0(\psi x-\psi t)$. From \cite{10.1093/biomet/asz020}, $\exists$ a constant $C > 0$ with
\begin{equation}
\mathbf{1}_{x\leq t}\leq G_{\psi,t}(x)\leq \mathbf{1}_{x\leq t+1/\psi},\ \ \sup_{x,t}\vert G_{\psi,t}^{'}(x)\vert\leq C\psi,\ \ \sup_{x,t}\vert G_{\psi,t}^{''}(x)\vert\leq C\psi^2,\ \ \sup_{x,t}\vert G_{\psi,t}^{'''}(x)\vert\leq C\psi^3
\label{pGs}
\end{equation}
$\forall \psi>0$, we define the set $\mathcal{A}_\psi = \{G_{\psi,t}|t\in\mathbf{R}\}$. $\forall \delta>0$, choose $\gamma = \delta/(C\psi)$, then $\forall G_{\psi,t}\in \mathcal{A}_\psi, x,y\in\mathbf{R}$ with $\vert x-y\vert<\gamma$, $\vert G_{\psi,t}(x)-G_{\psi,t}(y)\vert\leq C\psi\vert x-y\vert<\delta\Rightarrow \mathcal{A}_\psi$ is equi-continuous and uniformly bounded by $1$, from theorem 3.1 in \cite{10.2307/2237541},
\begin{equation}
\lim_{n\to\infty}\sup_{G_{\psi,t}\in\mathcal{A}_\psi}\vert\mathbf{E}G_{\psi,t}\left(\widetilde{\mathcal{M}}_{m,n}(x) - \widetilde{\mathcal{M}}^-_{m,n}(1-x)\right) - \mathbf{E}G_{\psi,t}\left(\mathcal{M}_m(x) - \mathcal{M}_m(1-x)\right)\vert = 0
\end{equation}
for any fixed $x\in[\frac{1}{2}+c, 1-c]$. From \eqref{pGs},
\begin{equation}
\begin{aligned}
\mathbf{P}(\widetilde{\mathcal{M}}_{m,n}(x) - \widetilde{\mathcal{M}}_{m,n}^-(1-x)\leq z) - \mathbf{P}(\mathcal{M}_m(x) - \mathcal{M}_m(1-x)\leq z)\\
\leq \mathbf{E}G_{\psi,z}\left(\widetilde{\mathcal{M}}_{m,n}(x) - \widetilde{\mathcal{M}}_{m,n}^-(1-x)\right) -\mathbf{E}G_{\psi,z-1/\psi}(\mathcal{M}_m(x) - \mathcal{M}_m(1-x))\\
\leq \sup_{G_{\psi,t}\in\mathcal{A}_\psi}\vert\mathbf{E}G_{\psi,t}\left(\widetilde{\mathcal{M}}_{m,n}(x) - \widetilde{\mathcal{M}}^-_{m,n}(1-x)\right) - \mathbf{E}G_{\psi,t}\left(\mathcal{M}_m(x) - \mathcal{M}_m(1-x)\right)\vert + \mathbf{P}\left(z-\frac{1}{\psi}<\mathcal{M}_m(x) - \mathcal{M}_m(1-x)\leq z+\frac{1}{\psi}\right)\\
\mathbf{P}(\widetilde{\mathcal{M}}_{m,n}(x) - \widetilde{\mathcal{M}}_{m,n}^-(1-x)\leq z) - \mathbf{P}(\mathcal{M}_m(x) - \mathcal{M}_m(1-x)\leq z)\\
\geq \mathbf{E}G_{\psi,z-1/\psi}\left(\widetilde{\mathcal{M}}_{m,n}(x) - \widetilde{\mathcal{M}}^-_{m,n}(1-x)\right) - \mathbf{E}G_{\psi,z}\left(\mathcal{M}_m(x) - \mathcal{M}_m(1-x)\right)\\
\geq -\sup_{G_{\psi,t}\in\mathcal{A}_\psi}\vert\mathbf{E}G_{\psi,t}\left(\widetilde{\mathcal{M}}_{m,n}(x) - \widetilde{\mathcal{M}}^-_{m,n}(1-x)\right) - \mathbf{E}G_{\psi,t}\left(\mathcal{M}_m(x) - \mathcal{M}_m(1-x)\right)\vert - \mathbf{P}\left(z-\frac{1}{\psi}<\mathcal{M}_m(x) - \mathcal{M}_m(1-x)\leq z+\frac{1}{\psi}\right)\\
\end{aligned}
\end{equation}
Choose $y = x, z = z+\frac{1}{\psi}, w = z-\frac{1}{\psi}$ in \eqref{firProb}, and let $\psi\to\infty$,
\begin{equation}
\lim_{n\to\infty}\sup_{z\in\mathbf{R}}\vert \mathbf{P}(\widetilde{\mathcal{M}}_{m,n}(x) - \widetilde{\mathcal{M}}_{m,n}^-(1-x)\leq z) - \mathbf{P}(\mathcal{M}_m(x) - \mathcal{M}_m(1-x)\leq z)\vert = 0
\label{thrd}
\end{equation}
Finally, for any given $\xi>0$, we choose $\frac{1}{2}+c = x_0<x_1<...<x_M=1-c$ and $x_i - x_{i-1}<\delta,i = 1,2,...,M$ with sufficiently small $\delta>0$, $\forall x\in[\frac{1}{2}+c,1-c],\exists I\in\{0,1,...,M\}$ such that $\vert x-x_I\vert<\delta$, and
\begin{equation}
\begin{aligned}
\sup_{z\in\mathbf{R}}\vert\mathbf{P}(\widetilde{\mathcal{M}}_{m,n}(x) - \widetilde{\mathcal{M}}_{m,n}^-(1-x)\leq z) - \mathbf{P}(\mathcal{M}_m(x) - \mathcal{M}_m(1-x)\leq z)\vert\\
\leq\sup_{z\in\mathbf{R}}\vert \mathbf{P}(\widetilde{\mathcal{M}}_{m,n}(x) - \widetilde{\mathcal{M}}_{m,n}^-(1-x)\leq z) - \mathbf{P}(\widetilde{\mathcal{M}}_{m,n}(x_I) - \widetilde{\mathcal{M}}_{m,n}^-(1-x_I)\leq z)\vert\\
+\max_{I = 1,2,...,M}\sup_{z\in\mathbf{R}}\vert \mathbf{P}(\widetilde{\mathcal{M}}_{m,n}(x_I) - \widetilde{\mathcal{M}}_{m,n}^-(1-x_I)\leq z) - \mathbf{P}(\mathcal{M}_m(x_I) - \mathcal{M}_m(1-x_I)\leq z)\vert\\
+\sup_{z\in\mathbf{R}}\vert\mathbf{P}(\mathcal{M}_m(x_I) - \mathcal{M}_m(1-x_I)\leq z) - \mathbf{P}(\mathcal{M}_m(x) - \mathcal{M}_m(1-x)\leq z)\vert
\end{aligned}
\end{equation}
From \eqref{Deltas}, $\forall \xi>0$,
\begin{equation}
\begin{aligned}
\sup_{z\in\mathbf{R}}\vert \mathbf{P}(\widetilde{\mathcal{M}}_{m,n}(x) - \widetilde{\mathcal{M}}_{m,n}^-(1-x)\leq z) - \mathbf{P}(\widetilde{\mathcal{M}}_{m,n}(x_I) - \widetilde{\mathcal{M}}_{m,n}^-(1-x_I)\leq z)\vert\leq \mathbf{P}\left(\vert \widetilde{\mathcal{M}}_{m,n}^-(1-x) - \widetilde{\mathcal{M}}_{m,n}^-(1-x_I)\vert>\frac{\xi}{2}\right)\\
+\mathbf{P}\left(\vert \widetilde{\mathcal{M}}_{m,n}(x) - \widetilde{\mathcal{M}}_{m,n}(x_I)\vert>\frac{\xi}{2}\right)
+ 2\max_{I = 1,2,...,M}\sup_{z\in\mathbf{R}}\vert \mathbf{P}(\widetilde{\mathcal{M}}_{m,n}(x_I) - \widetilde{\mathcal{M}}_{m,n}^-(1-x_I)\leq z) - \mathbf{P}(\mathcal{M}_m(x_I) - \mathcal{M}_m(1-x_I)\leq z)\vert\\
+\sup_{z\in\mathbf{R}}\mathbf{P}(z-\xi<\mathcal{M}_m(x_I) - \mathcal{M}_m(1-x_I)\leq z+\xi)
\end{aligned}
\end{equation}
Since $\sup_{x\in[\frac{1}{2}+c,1-c]}\mathbf{P}\left(\vert \widetilde{\mathcal{M}}_{m,n}(x) - \widetilde{\mathcal{M}}_{m,n}(x_I)\vert>\frac{\xi}{2}\right)$ and $\sup_{x\in[\frac{1}{2}+c,1-c]}\mathbf{P}\left(\vert \widetilde{\mathcal{M}}_{m,n}^-(1-x) - \widetilde{\mathcal{M}}_{m,n}^-(1-x_I)\vert>\frac{\xi}{2}\right)$ are less or equal to $ \mathbf{P}\left(\sup_{y,z\in[0,1],\vert y - z\vert<\delta}\vert\widetilde{\mathcal{M}}_{m,n}(y) - \widetilde{\mathcal{M}}_{m,n}(z)\vert>\frac{\xi}{2}\right)$, from \eqref{approxConverge}, \eqref{firProb}, and \eqref{thrd}, we prove \eqref{SecRes}.
\end{proof}

In the second lemma, we define $\widehat{\alpha}(x) = \frac{1}{\sqrt{n}}\sum_{i = 1}^n(\mathbf{1}_{\widehat{\epsilon}_i\leq x} - F(x))$, and $\widetilde{\alpha}^*(x) = \frac{1}{\sqrt{n}}\sum_{i = 1}^n(\mathbf{1}_{\epsilon^*_i\leq x} - \widehat{F}(x))$. $\widehat{\epsilon}_i,i=1,2,...,n$ and $\widehat{F}$ are defined in \eqref{DefResi}, and $\epsilon^*_i, i = 1,2,...,n$ are i.i.d. random variables generated from $\widehat{F}$. $\epsilon^*_i$ serves as bootstrapped residuals in algorithm \ref{alg1}. We define $\widetilde{F}(x) = \frac{1}{n}\sum_{i = 1}^n\mathbf{1}_{\epsilon_i\leq x}$ and $\widetilde{\alpha}(x) = \frac{1}{\sqrt{n}}\sum_{i=1}^n(\mathbf{1}_{\epsilon_i\leq x} - F(x)),\forall x\in\mathbf{R}$. The notation $O_p$ and $o_p$ have the same meaning as definition 1.9, \cite{mathStat}.

\begin{lemma}
Suppose assumption 1. to 4. hold true. Then for any given $\xi>0$, $-\infty<r\leq s <\infty$, $\exists \delta>0$ such that
\begin{equation}
\limsup_{n\to\infty}\mathbf{P}\left(\sup_{x,y\in[r,s],\vert x-y\vert<\delta}\vert\widehat{\alpha}(x) - \widehat{\alpha}(y)\vert > \xi\right)<\xi
\label{Half}
\end{equation}
Besides, $\exists \delta>0$ and $N>0$, $\forall n\geq N$,
\begin{equation}
\mathbf{P}\left(\left\{\mathbf{P}^*\left(\sup_{x,y\in[r,s],\vert x - y\vert<\delta}\vert \widetilde{\alpha}^*(x) - \widetilde{\alpha}^*(y)\vert>\xi\right)>\xi\right\}\right)<\xi
\label{Quan_Boot}
\end{equation}
\label{lemmaBoot}
\end{lemma}

\begin{proof}[proof of lemma \ref{lemmaBoot}]
From assumption 4., $F$ is strictly increasing in $\mathbf{R}$. From lemma 4.1 and 4.2, \cite{10.2307/2240410}, $\exists$ independent random variables $U_i,i=1,2,...$ with uniform distribution on $[0,1]$, a Brownian bridge $B$, and a constant $C$ such that
\begin{equation}
\mathbf{P}\left(\sup_{x\in[0,1]}\vert \frac{1}{\sqrt{n}}\sum_{i=1}^n(\mathbf{1}_{U_i\leq x} - x) - B(x)\vert\geq C\log(n)/\sqrt{n}\right)\leq C\log(n)/\sqrt{n}
\label{Ui}
\end{equation}
and $\forall 0<\delta <1/2$, $\xi > 0$,
\begin{equation}
\mathbf{E}\sup_{x,y\in[0,1], \vert x - y\vert<\delta}\vert B(x) - B(y)\vert\leq C(-\delta\log(\delta))^{1/2}\Rightarrow \mathbf{P}\left(\sup_{x,y\in[0,1], \vert x - y\vert<\delta}\vert B(x) - B(y)\vert > \xi\right)\leq \frac{C(-\delta\log(\delta))^{1/2}}{\xi}
\label{ExpBrown}
\end{equation}
We choose $\epsilon_i = F^{-1}(U_i),i=1,2,...,n$($\epsilon_i$ has distribution $F$ according to page 150, \cite{billing}),
\begin{equation}
\begin{aligned}
\widetilde{\alpha}(x) = \frac{1}{\sqrt{n}}(\sum_{i=1}^n \mathbf{1}_{U_i\leq F(x)} - F(x))\Rightarrow \mathbf{P}\left(\sup_{x\in \mathbf{R}}\vert \widetilde{\alpha}(x)- B(F(x))\vert\geq C\log(n)/\sqrt{n}\right)\leq C\log(n)/\sqrt{n}
\end{aligned}
\label{ExpAlpha}
\end{equation}
From assumption 3., $\max_{i=1,...,n}\frac{1}{n}x_i^T(X^TX/n)^{-1}x_i\leq \frac{1}{n}M^2\Vert(X^TX/n)^{-1}\Vert_2 = O(1/n)$; $\mathbf{E}\Vert (X^TX)^{1/2}(\widehat{\beta}-\beta)\Vert_2^2 = p\sigma^2$, which implies $\Vert (X^TX)^{1/2}(\widehat{\beta}-\beta)\Vert_2 = O_p(1)$; $\mathbf{E}\widehat{\lambda}^2 \leq 2\mathbf{E}(\frac{1}{n}\sum_{i=1}^n\epsilon_i)^2 + 2\mathbf{E}(\overline{x}_n^T(\widehat{\beta} - \beta))^2 = \frac{2\sigma^2}{n} + \frac{2\sigma^2 \overline{x}_n(X^TX/n)^{-1}\overline{x}_n}{n}\Rightarrow \widehat{\lambda} = O_p(1/\sqrt{n})$. We define $\widehat{\alpha}^{'}(z) = \sqrt{n}\left(\frac{1}{n}\sum_{i=1}^n\mathbf{1}_{\widehat{\epsilon}_i^{'}\leq z}-F(z)\right)$. From theorem 6.2.1 in \cite{empirical_process}, $\sup_{x\in\mathbf{R}}\vert\widehat{\alpha}^{'}(x) - \widetilde{\alpha}(x)-\sqrt{n}F^{'}(x)\overline{x}_n^T(\widehat{\beta}-\beta)\vert = o_p(1)$. Therefore,
\begin{equation}
\begin{aligned}
\sup_{x\in\mathbf{R}}\vert\widehat{\alpha}(x) - \widetilde{\alpha}(x) - \sqrt{n}F^{'}(x)\overline{x}_n^T(\widehat{\beta} - \beta) - \sqrt{n}F^{'}(x)\widehat{\lambda}\vert\leq  \sup_{x\in\mathbf{R}}\vert\widehat{\alpha}^{'}(x) - \widetilde{\alpha}(x)-\sqrt{n}F^{'}(x)\overline{x}_n^T(\widehat{\beta}-\beta)\vert\\
+\sup_{x\in\mathbf{R}}\vert\widetilde{\alpha}(x+\widehat{\lambda}) - \widetilde{\alpha}(x)\vert+\sup_{x\in\mathbf{R}}\sqrt{n}\vert(F^{'}(x+\widehat{\lambda}) - F^{'}(x))\overline{x}_n^T(\widehat{\beta} - \beta)\vert +\sup_{x\in\mathbf{R}}\sqrt{n}\vert F(x+\widehat{\lambda}) - F(x) - F^{'}(x)\widehat{\lambda}\vert
\end{aligned}
\label{Convert}
\end{equation}
From assumption 1. and 3., and Taylor's theorem, $\sup_{x\in\mathbf{R}}\sqrt{n}\vert(F^{'}(x+\widehat{\lambda}) - F^{'}(x))\overline{x}_n^T(\widehat{\beta} - \beta)\vert$ and $\sup_{x\in\mathbf{R}}\sqrt{n}\vert F(x+\widehat{\lambda}) - F(x) - F^{'}(x)\widehat{\lambda}\vert$ have order $O_p(1/\sqrt{n})$. From \eqref{ExpAlpha}, with probability tending to $1$,
\begin{equation}
\sup_{x\in\mathbf{R}}\vert\widetilde{\alpha}(x+\widehat{\lambda}) - \widetilde{\alpha}(x)\vert\leq\frac{2C\log(n)}{\sqrt{n}} + \sup_{x\in\mathbf{R}}\vert B(F(x+\widehat{\lambda})) - B(F(x))\vert
\end{equation}
$F$ is uniform continuous according to assumption 1., so $\sup_{x\in\mathbf{R}}\vert\widetilde{\alpha}(x+\widehat{\lambda}) - \widetilde{\alpha}(x)\vert = o_p(1)\Rightarrow \sup_{x\in\mathbf{R}}\vert\widehat{\alpha}(x) - \widetilde{\alpha}(x) - \sqrt{n}F^{'}(x)\overline{x}_n^T(\widehat{\beta} - \beta) - \sqrt{n}F^{'}(x)\widehat{\lambda}\vert = o_p(1)$. For any given $-\infty<r\leq s<\infty$ and sufficiently small $\delta>0$,
\begin{equation}
\begin{aligned}
\sup_{x,y\in[r,s],\vert x - y\vert<\delta}\vert\widehat{\alpha}(x) - \widehat{\alpha}(y)\vert\leq \sup_{x,y\in[r,s],\vert x - y\vert<\delta}\vert\widetilde{\alpha}(x) - \widetilde{\alpha}(y)\vert + \sup_{x,y\in[r,s],\vert x - y\vert<\delta}\sqrt{n}\vert(F^{'}(x) - F^{'}(y))\times(\overline{x}_n^T(\widehat{\beta} - \beta)+\widehat{\lambda})\vert +o_p(1)
\end{aligned}
\end{equation}

From assumption 1., \eqref{ExpAlpha} and \eqref{ExpBrown}, we prove \eqref{Half}.

We define the function $\widehat{\phi}(x) = \inf\{t| x\leq \widehat{F}(t)\},x\in[0,1]$. Page 150, \cite{billing} implies $\widehat{\phi}(x)\leq t \Leftrightarrow x\leq \widehat{F}(t)$. If $U$ has uniform distribution on $[0,1]$, then $\widehat{\phi}(U)$ has distribution $\widehat{F}$. We choose $\epsilon^*_i = \widehat{\phi}(U_i), i = 1,2,...,n$; \eqref{Ui} implies
\begin{equation}
\widetilde{\alpha}^*(x) = \frac{1}{\sqrt{n}}\sum_{i = 1}^n(\mathbf{1}_{U_i\leq \widehat{F}(x)} - \widehat{F}(x))\Rightarrow \mathbf{P}^*\left(\sup_{x\in \mathbf{R}}\vert \widetilde{\alpha}^*(x)- B(\widehat{F}(x))\vert\geq C\log(n)/\sqrt{n}\right)\leq C\log(n)/\sqrt{n}
\label{Secc}
\end{equation}
From assumption 3., $\widehat{F}(x) = \frac{1}{n}\sum_{i = 1}^n\mathbf{1}_{\epsilon_i\leq x + x_i^T(\widehat{\beta} - \beta)}\leq \frac{1}{n}\sum_{i = 1}^n\mathbf{1}_{\epsilon_i\leq x + M\Vert\widehat{\beta} - \beta\Vert_2} = \widetilde{F}(x+ M\Vert\widehat{\beta} - \beta\Vert_2)$; and $\widehat{F}(x)\geq \widetilde{F}(x - M\Vert\widehat{\beta} - \beta\Vert_2)$. For any given $\varepsilon>0$, we can find $C_\varepsilon>0$ with $\mathbf{P}(\Vert\widehat{\beta} - \beta\Vert_2 > C_\varepsilon/\sqrt{n})<\varepsilon$ for any $n$. From Glivenko - Cantelli theorem and dominated convergence theorem, $\lim_{n\to\infty}\mathbf{P}(\sup_{x\in\mathbf{R}}\vert\widetilde{F}(x) - F(x)\vert > \varepsilon) = 0$. If $\Vert\widehat{\beta} - \beta\Vert_2\leq C_\varepsilon/\sqrt{n}$ and $\sup_{x\in\mathbf{R}}\vert\widetilde{F}(x) - F(x)\vert \leq \varepsilon$, for any given $-\infty < r\leq s < \infty$, $\delta >0$, $-\varepsilon+ F(x-\frac{MC_\varepsilon}{\sqrt{n}})\leq \widehat{F}(x) \leq \varepsilon + F(x+\frac{MC_\varepsilon}{\sqrt{n}})$, and $\sup_{r\leq x\leq y\leq s, y-x<\delta}\widehat{F}(y) - \widehat{F}(x)\leq 2\varepsilon + \sup_{r\leq x\leq y\leq s, y-x<\delta} F(y + \frac{MC_\varepsilon}{\sqrt{n}}) - F(x - \frac{MC_\varepsilon}{\sqrt{n}})$. For any given $-\infty<r\leq s <\infty$ and $\xi>0$, we choose sufficiently small $\varepsilon, \delta>0$ and define $\zeta = 2\varepsilon + \sup_{r\leq x\leq y\leq s, y-x<\delta} F(y + \frac{MC_\varepsilon}{\sqrt{n}}) - F(x - \frac{MC_\varepsilon}{\sqrt{n}})$,
\begin{equation}
\begin{aligned}
\sup_{x,y\in[r,s],\vert x-y\vert<\delta}\vert \widetilde{\alpha}^*(x) - \widetilde{\alpha}^*(y)\vert\leq 2\sup_{x\in \mathbf{R}}\vert \widetilde{\alpha}^*(x)- B(\widehat{F}(x))\vert + \sup_{x,y\in[r,s],\vert x-y\vert<\delta}\vert B(\widehat{F}(x)) - B(\widehat{F}(y))\vert\\
\leq 2\sup_{x\in \mathbf{R}}\vert \widetilde{\alpha}^*(x)- B(\widehat{F}(x))\vert + \sup_{x,y\in[-\varepsilon+F(r-\frac{MC_\varepsilon}{\sqrt{n}}),\varepsilon+F(s+\frac{MC_\varepsilon}{\sqrt{n}})],\vert x - y\vert\leq \zeta}\vert B(x) - B(y)\vert\\
\Rightarrow \mathbf{P}^*\left(\sup_{x,y\in[r,s],\vert x-y\vert<\delta}\vert \widetilde{\alpha}^*(x) - \widetilde{\alpha}^*(y)\vert>\xi\right)\leq \mathbf{P}^*\left(\sup_{x\in \mathbf{R}}\vert \widetilde{\alpha}^*(x)- B(\widehat{F}(x))\vert>\frac{\xi}{4}\right)
+ \mathbf{P}^*\left(\sup_{x,y\in[0,1],\vert x - y\vert\leq \zeta}\vert B(x) - B(y)\vert>\frac{\xi}{2}\right)
\end{aligned}
\end{equation}
For $F$ is uniform continuous, \eqref{ExpBrown} and \eqref{Secc} imply \eqref{Quan_Boot}.
\end{proof}

\begin{proof}[proof of theorem \ref{theoPd}]
According to theorem 13.5 in \cite{billing}, it suffices to show $\forall z_1,...,z_k\in[0,1]$, $(\widetilde{\mathcal{M}}_m(z_1),...,\widetilde{\mathcal{M}}_m(z_k))\to_{\mathcal{L}} (\mathcal{M}_m(z_1),...,\mathcal{M}_m(z_k))$ in $\mathbf{R}^k$; $\mathcal{M}_m(1)-\mathcal{M}_m(1-\delta)\to_{\mathcal{L}}0$ in $\mathbf{R}$ as $\delta\to 0,\delta>0$; and $\exists \beta\geq 0,\alpha>1/2$, and a non-decreasing, continuous function $G$ on $[0,1]$ such that
\begin{equation}
\mathbf{E}\vert \widetilde{\mathcal{M}}_m(t)-\widetilde{\mathcal{M}}_m(s)\vert^{2\beta}\vert\widetilde{\mathcal{M}}_m(s)-\widetilde{\mathcal{M}}_m(r)\vert^{2\beta}\leq (G(t) - G(r))^{2\alpha}\ \forall 1\geq t>s>r\geq 0
\label{Con3}
\end{equation}
The first condition: we define $c^T = (c_1,...,c_n) = x_f^T(X^TX)^{-1}X^T-\frac{1}{n}e^T\Rightarrow c_i = x_f^T(X^TX)^{-1}x_i-\frac{1}{n}$, and $z_j^{'} = 2mz_j - m, j = 1,2,...,k$.
\begin{equation}
\begin{aligned}
\sum_{j=1}^ks_j\widetilde{\mathcal{M}}_m(z_j) = \sum_{i=1}^n\left((\sum_{j=1}^k\sqrt{n}s_jF^{'}(z_j^{'}))c_i\epsilon_i-\frac{1}{\sqrt{n}}(\sum_{j=1}^ks_j(\mathbf{1}_{\epsilon_i\leq z_j^{'}}-F(z_j^{'})))\right)\Rightarrow \mathbf{E}\sum_{j=1}^ks_j\widetilde{\mathcal{M}}_m(z_j) = 0\\
\label{crus}
\end{aligned}
\end{equation}
Form assumption 3. and (5.8.4) in \cite{10.5555/2422911}, we define $Y_i = (\sum_{j=1}^k\sqrt{n}s_jF^{'}(z_j^{'}))c_i\epsilon_i-\frac{1}{\sqrt{n}}(\sum_{j=1}^ks_j(\mathbf{1}_{\epsilon_i\leq z_j^{'}}-F(z_j^{'})))$,
\begin{equation}
\begin{aligned}
\mathbf{E}Y^2_i
= n\sigma^2c_i^2\sum_{j=1}^k\sum_{l=1}^k s_js_lF^{'}(z_j^{'})F^{'}(z_l^{'})
+\frac{1}{n}\sum_{j=1}^k\sum_{l=1}^ks_js_l(F(\min(z_j^{'},z_l^{'}))-F(z_j^{'})F(z_l^{'})) - 2c_i\sum_{j=1}^k\sum_{l=1}^ks_js_lF^{'}(z_j^{'})H(z_l^{'})\\
\Rightarrow \lim_{n\to\infty}\sum_{i=1}^n\mathbf{E}Y^2_i = \lim_{n\to\infty}\sigma^2\sum_{j=1}^k\sum_{l=1}^k s_js_lF^{'}(z_j^{'})F^{'}(z_l^{'})\times \left(x_f^T(\frac{X^TX}{n})^{-1}x_f + 1 - 2x_f^T(\frac{X^TX}{n})^{-1}\overline{x}_n\right)\\
+\sum_{j=1}^k\sum_{l=1}^ks_js_l(F(\min(z_j^{'},z_l^{'}))-F(z_j^{'})F(z_l^{'}))
- \lim_{n\to\infty}2\sum_{j=1}^k\sum_{l=1}^ks_js_lF^{'}(z_j^{'})H(z_l^{'})\times (x_f^T(\frac{X^TX}{n})^{-1}\overline{x}_n-1)\\
=\sigma^2\mathcal{K}(x^T_fA^{-1}x_f + 1 - 2x_f^TA^{-1}b) + \mathcal{N} - 2\mathcal{R}(x^T_fA^{-1}b - 1)
\end{aligned}
\label{Fifty_Eight}
\end{equation}
Here we define $\mathcal{K} = \sum_{j=1}^k\sum_{l=1}^k s_js_lF^{'}(z_j^{'})F^{'}(z_l^{'})$, $\mathcal{N} = \sum_{l=1}^ks_js_l(F(\min(z_j^{'},z_l^{'}))-F(z_j^{'})F(z_l^{'}))$, and $\mathcal{R} = \sum_{j=1}^k\sum_{l=1}^ks_js_lF^{'}(z_j^{'})H(z_l^{'})
$.
From mean value inequality,
\begin{equation}
\begin{aligned}
\sum_{i=1}^n\mathbf{E}\vert Y_i\vert^3
\leq 4k^2\mathbf{E}\vert\epsilon_1\vert^3\sum_{j = 1}^k\vert s_jF^{'}(z_j^{'})\vert^3\times n\sqrt{n}\sum_{i = 1}^n\vert c_i\vert^3 + 4k^2\sum_{j = 1}^k\vert s_j\vert^3\times \frac{1}{n\sqrt{n}}\sum_{i = 1}^n\mathbf{E}\vert \mathbf{1}_{\epsilon_i\leq z_j^{'}}-F(z_j^{'})\vert^3
\end{aligned}
\end{equation}
From assumption 3.,
\begin{equation}
n\sqrt{n}\sum_{i=1}^n\vert c_i\vert^3\leq n\sqrt{n}\max_{i=1,2,...,n}\vert c_i\vert\times \sum_{i=1}^n c_i^2\leq \frac{1+M^2\Vert(X^TX/n)^{-1}\Vert_2}{\sqrt{n}}\times  \left(x_f^T(\frac{X^TX}{n})^{-1}x_f + 1 - 2x_f^T(\frac{X^TX}{n})^{-1}\overline{x}_n\right)
\end{equation}
which has order $O(1/\sqrt{n})$. $\Vert(X^TX/n)^{-1}\Vert_2$ is the matrix 2 norm of $(X^TX/n)^{-1}$. If $\sigma^2 \mathcal{K}\times (x_f^TA^{-1}x_f+1-2x_f^TA^{-1}b) + \mathcal{N} - 2\mathcal{R}(x_f^TA^{-1}b-1)\neq 0$, from Theorem 1.15, Theorem 1.11, and (1.97) in \cite{mathStat},
\begin{equation}
\sum_{j=1}^ks_j\widetilde{\mathcal{M}}_m(z_j) = \frac{\sum_{j=1}^ks_j\widetilde{\mathcal{M}}_m(z_j)}{\sqrt{\sum_{i=1}^n\mathbf{E}Y^2_i}}\times \sqrt{\sum_{i=1}^n\mathbf{E}Y^2_i}\to_{\mathcal{L}} N(0,\sigma^2 \mathcal{K}\times (x_f^TA^{-1}x_f+1-2x_f^TA^{-1}b) + \mathcal{N} - 2\mathcal{R}(x_f^TA^{-1}b-1))
\end{equation}
On the other hand, if $\sigma^2 \mathcal{K}\times (x_f^TA^{-1}x_f+1-2x_f^TA^{-1}b) + \mathcal{N} - 2\mathcal{R}(x_f^TA^{-1}b-1)=0$, then $\forall \delta>0$, from \eqref{Fifty_Eight}, $\lim_{n\to\infty}\mathbf{P}(\vert \sum_{j=1}^ks_j\widetilde{\mathcal{M}}_m(z_j)\vert\geq \delta)\leq \lim_{n\to\infty}\frac{\mathbf{E}\vert \sum_{j=1}^ks_j\widetilde{\mathcal{M}}_m(z_j)\vert^2}{\delta^2} =0\Rightarrow \sum_{j=1}^ks_j\widetilde{\mathcal{M}}_m(z_j)\to_{\mathcal{L}}0$. From theorem 1.9, \cite{mathStat}, we prove the first condition.

The second condition: $\forall \xi>0$,
\begin{equation}
\begin{aligned}
\mathbf{P}\left(\vert \mathcal{M}_m(1)-\mathcal{M}_m(1-\delta)\vert\geq \xi\right)\leq \frac{\mathbf{E}\vert \mathcal{M}_m(1)-\mathcal{M}_m(1-\delta)\vert^2}{\xi^2} \leq \frac{\sigma^2(x_f^TA^{-1}x_f-2x_f^TA^{-1}b+1)(F^{'}(m)-F^{'}(m-2m\delta))^2}{\xi^2}\\
+\frac{2\vert x_f^TA^{-1}b-1\vert\times \vert F^{'}(m)-F^{'}(m-2m\delta)\vert\times\vert H(m)-H(m-2m\delta)\vert}{\xi^2} + \frac{\vert F(m)-F(m-2m\delta)\vert+\vert F(m)-F(m-2m\delta)\vert^2}{\xi^2}
\end{aligned}
\end{equation}
From assumption 1., $\lim_{\delta\to 0,\delta>0}Prob\left(\vert \mathcal{M}_m(1)-\mathcal{M}_m(1-\delta)\vert\geq \xi\right) = 0$, and we prove the second condition.

The third condition: we choose $\beta = \alpha = 1$, and define $t^{'} = 2mt - m, \forall t$. For $\forall t, s\in[0,1]$, we define $\mathcal{A}(t,s) = \sqrt{n}\left(F^{'}(t^{'})-F^{'}(s^{'})\right)\times\left(x_f^T(X^TX)^{-1}X^T\epsilon-\frac{1}{n}\sum_{i=1}^n\epsilon_i\right)$, and $\mathcal{B}(t,s) =   \frac{1}{\sqrt{n}}\sum_{i=1}^n(\mathbf{1}_{s^{'}<\epsilon_i\leq t^{'}} - F(t^{'}) + F(s^{'}))$. From mean value inequality,
\begin{equation}
\begin{aligned}
\mathbf{E}\vert \widetilde{\mathcal{M}}_m(t)-\widetilde{\mathcal{M}}_m(s)\vert^{2}\vert\widetilde{\mathcal{M}}_m(s)-\widetilde{\mathcal{M}}_m(r)\vert^{2}\leq 4\mathbf{E}\left(\mathcal{A}(t,s)^2\mathcal{A}(s,r)^2 + \mathcal{B}(t,s)^2\mathcal{A}(s,r)^2 + \mathcal{A}(t,s)^2\mathcal{B}(s,r)^2 + \mathcal{B}(t,s)^2\mathcal{B}(s,r)^2\right)
\end{aligned}
\end{equation}
From assumption 3), $x_f^T(X^TX/n)^{-1}x_f\to x_f^TA^{-1}x_f$ and $x_f^T(X^TX/n)^{-1}\overline{x}_n\to x_f^TA^{-1}b$. Therefore, $\exists C>0$ such that $\vert x_f^T(\frac{X^TX}{n})^{-1}x_f + 1 - 2x_f^T(\frac{X^TX}{n})^{-1}\overline{x}_n\vert\leq C, \forall n$. Define $c=(c_1,...,c_n)^T$, $c_i = x_f^T(X^TX)^{-1}x_i-1/n$, $\forall 1\geq t>s>r\geq 0$
\begin{equation}
\begin{aligned}
\mathbf{E}\mathcal{A}(t,s)^2\mathcal{A}(s,r)^2 = n^2\left(F^{'}(t^{'})-F^{'}(s^{'})\right)^2\left(F^{'}(s^{'})-F^{'}(r^{'})\right)^2\mathbf{E}(c^T\epsilon)^4\\
= n^2\left(F^{'}(t^{'})-F^{'}(s^{'})\right)^2\left(F^{'}(s^{'})-F^{'}(r^{'})\right)^2
\times \left(\mathbf{E}\epsilon_1^4\times\sum_{i=1}^nc_i^4 + 3\sigma^4\sum_{i=1}^n\sum_{j=1,j\neq i}^nc_i^2c_j^2\right)\\
\leq 16m^4(\mathbf{E}\epsilon_1^4+3\sigma^4)(t-s)^2(s-r)^2\times \left(x_f^T(\frac{X^TX}{n})^{-1}x_f + 1 - 2x_f^T(\frac{X^TX}{n})^{-1}\overline{x}_n\right)^2\times \sup_{x\in\mathbf{R}}\vert F^{''}(x)\vert^4\\
\leq 16C^2m^4(\mathbf{E}\epsilon_1^4+3\sigma^4)(t-r)^2\times \sup_{x\in\mathbf{R}}\vert F^{''}(x)\vert^4
\end{aligned}
\label{con31}
\end{equation}

\begin{equation}
\begin{aligned}
\mathbf{E}\mathcal{B}(t,s)^2\mathcal{B}(s,r)^2\leq \frac{1}{n}\mathbf{E}(\mathbf{1}_{s^{'}<\epsilon_1\leq t^{'}} - F(t^{'}) + F(s^{'}))^2(\mathbf{1}_{r^{'}<\epsilon_1\leq s^{'}} - F(s^{'}) + F(r^{'}))^2\\
+\mathbf{E}(\mathbf{1}_{s^{'}<\epsilon_1\leq t^{'}} - F(t^{'}) + F(s^{'}))^2\times\mathbf{E}(\mathbf{1}_{r^{'}<\epsilon_1\leq s^{'}} - F(s^{'}) + F(r^{'}))^2+2\left(\mathbf{E}(\mathbf{1}_{s^{'}<\epsilon_1\leq t^{'}} - F(t^{'}) + F(s^{'}))(\mathbf{1}_{r^{'}<\epsilon_1\leq s^{'}} - F(s^{'}) + F(r^{'}))\right)^2\\
\leq \frac{3}{n}(F(t^{'})-F(s^{'}))(F(s^{'})-F(r^{'}))
+ (F(t^{'})-F(s^{'}))(F(s^{'})-F(r^{'})) + 2(F(t^{'})-F(s^{'}))^2(F(s^{'})-F(r^{'}))^2\leq 6(F(t^{'})-F(r^{'}))^2
\end{aligned}
\label{con32}
\end{equation}

\begin{equation}
\begin{aligned}
\mathbf{E}\mathcal{A}(t,s)^2\mathcal{B}(s,r)^2 = \left(F^{'}(t^{'})-F^{'}(s^{'})\right)^2\times (\sum_{i=1}^nc_i^2\mathbf{E}\epsilon_i^2(\mathbf{1}_{r^{'}<\epsilon_i\leq s^{'}} - F(s^{'}) + F(r^{'}))^2\\
+\sigma^2(n-1)\sum_{i=1}^n c_i^2\mathbf{E}(\mathbf{1}_{r^{'}<\epsilon_1\leq s^{'}} - F(s^{'}) + F(r^{'}))^2
+ 2 \sum_{i=1}^n\sum_{j=1,j\neq i}^nc_ic_j\times (\mathbf{E}\epsilon_1\mathbf{1}_{r^{'}<\epsilon_1\leq s^{'}})^2
)\\
\leq \sigma^2\left(F^{'}(t^{'})-F^{'}(s^{'})\right)^2\times (n\sum_{i=1}^nc_i^2 + 2((\sum_{i=1}^nc_i)^2+\sum_{i=1}^nc_i^2))\\
\end{aligned}
\label{con33}
\end{equation}
Notice that $n\sum_{i=1}^nc_i^2 = x_f^T(\frac{X^TX}{n})^{-1}x_f + 1 - 2x_f^T(\frac{X^TX}{n})^{-1}\overline{x}_n$ and $\sum_{i=1}^n c_i = x_f^T(X^TX/n)^{-1}\overline{x}_n-1$, \eqref{Con3} is satisfied by choosing $G(x) = C^{'}x$ with a sufficiently large constant $C^{'}$. And we prove \eqref{AsymptoticM}.

We choose $m > s+1$ and define $\widehat{\alpha}(x) = \sqrt{n}(\widehat{F}(x) - F(x))$. $\forall x\in[r,s]$, from Taylor's theorem
\begin{equation}
\begin{aligned}
\mathcal{S}(x)
= \sqrt{n}(F(x + x_f^T(X^TX)^{-1}X^T\epsilon) - F(-x + x_f^T(X^TX)^{-1}X^T\epsilon)) \\
- \sqrt{n}\left(\mathbf{E}^*\left(\widehat{F}(x + x_f^T(X^TX)^{-1}X^T\epsilon^*) - \widehat{F}^-(-x + x_f^T(X^TX)^{-1}X^T\epsilon^*)\right)\right)\\
= \left(F^{'}(x) - F^{'}(-x)\right)\times \sqrt{n}x_f^T(X^TX)^{-1}X^T\epsilon + \frac{F^{''}(\eta_1) - F^{''}(\eta_2)}{2}\times \sqrt{n}(x_f^T(X^TX)^{-1}X^T\epsilon)^2\\
-\mathbf{E}^*\left(\widehat{\alpha}(x+x_f^T(X^TX)^{-1}X^T\epsilon^*) - \widehat{\alpha}(x)\right) +\mathbf{E}^*\left(\widehat{\alpha}^{-}(-x +x_f^T(X^TX)^{-1}X^T\epsilon^*) - \widehat{\alpha}^{-}(-x)\right) \\
- \widehat{\alpha}(x) + \widehat{\alpha}^{-}(-x) - \sqrt{n}\mathbf{E}^*\left(F(x+x_f^T(X^TX)^{-1}X^T\epsilon^*) - F(x)\right)
+\sqrt{n}\mathbf{E}^*(F(-x + x_f^T(X^TX)^{-1}X^T\epsilon^*) - F(-x))\\
\Rightarrow \sup_{x\in[r,s]}\vert\mathcal{S}(x) - \left(\widetilde{\mathcal{M}}_m(\frac{x+m}{2m}) - \widetilde{\mathcal{M}}_m^-(\frac{-x+m}{2m})\right)\vert
\leq \sqrt{n}(x_f^T(X^TX)^{-1}X^T\epsilon)^2\times \sup_{x\in\mathbf{R}}\vert F^{''}(x)\vert\\
+ \sup_{x\in[r,s]}\vert\mathbf{E}^*(\widehat{\alpha}(x+x_f^T(X^TX)^{-1}X^T\epsilon^*) - \widehat{\alpha}(x))\vert
+ \sup_{x\in[r,s]}\vert \mathbf{E}^*(\widehat{\alpha}^-(-x +x_f^T(X^TX)^{-1}X^T\epsilon^*) - \widehat{\alpha}^-(-x))\vert\\
+ \sup_{x\in[r,s]}\vert\widehat{\alpha}(x) - \widetilde{\alpha}(x) - \frac{F^{'}(x)}{\sqrt{n}}\sum_{i=1}^n\epsilon_i\vert
+ \sup_{x\in[r,s]}\vert\widehat{\alpha}^-(-x) - \widetilde{\alpha}^-(-x) - \frac{F^{'}(-x)}{\sqrt{n}}\sum_{i=1}^n\epsilon_i\vert
+ \mathbf{E}^*(x_f^T(X^TX)^{-1}X^T\epsilon^*)^2\times \sup_{x\in\mathbf{R}}\vert F^{''}(x)\vert
\end{aligned}
\label{DI}
\end{equation}
From lemma \ref{lemmaBoot}, for any given $\xi > 0$, $\exists 1/2 >\delta > 0$ such that for sufficiently large $n$, $\mathbf{P}\left(\sup_{x,y\in[-m,m], \vert x - y\vert < \delta}\vert\widehat{\alpha}(x) - \widehat{\alpha}(y)\vert\leq \xi\right)>1-\xi$. If $\sup_{x,y\in[-m,m], \vert x - y\vert < \delta}\vert\widehat{\alpha}(x) - \widehat{\alpha}(y)\vert\leq \xi$, then $\forall x\in[r,s]$,
\begin{equation}
\begin{aligned}
\vert \mathbf{E}^*(\widehat{\alpha}^-(-x +x_f^T(X^TX)^{-1}X^T\epsilon^*) - \widehat{\alpha}^-(-x))\vert,\ \vert\mathbf{E}^*(\widehat{\alpha}(x+x_f^T(X^TX)^{-1}X^T\epsilon^*) - \widehat{\alpha}(x))\vert\\
\leq \sqrt{n}\mathbf{P}^*(\vert x_f^T(X^TX)^{-1}X^T\epsilon^*\vert > \delta) + \xi\leq \xi+ \frac{\widehat{\sigma}^2}{\sqrt{n}\delta^2}x_f^T\left(\frac{X^TX}{n}\right)^{-1}x_f\\
\end{aligned}
\label{continuityAlpha}
\end{equation}
Since $\mathbf{E}^*(x_f^T(X^TX)^{-1}X^T\epsilon^*)^2 = \frac{\widehat{\sigma}^2}{n}\left(\frac{x_f^T(X^TX)^{-1}x_f}{n}\right)^{-1}$ and $\mathbf{E}\widehat{\sigma}^2\leq 4\sigma^2 + \frac{4\sigma^2M^2}{n}\Vert(\frac{X^TX}{n})^{-1}\Vert_2+2\mathbf{E}\widehat{\lambda}^2 = O(1)$, combine with \eqref{Convert} we have $\forall \xi>0$,
\begin{equation}
\mathbf{P}\left(\sup_{x\in[r,s]}\vert\mathcal{S}(x) - \left(\widetilde{\mathcal{M}}_m(\frac{x+m}{2m}) - \widetilde{\mathcal{M}}_m^-(\frac{-x+m}{2m})\right)\vert > \xi\right)\to 0
\label{S_and_Mm}
\end{equation}
Finally, from \eqref{Deltas} and corollary \ref{coroLink}, $\forall \delta>0$,
\begin{equation}
\begin{aligned}
\sup_{x\in[r,s],y\in\mathbf{R}}\vert\mathbf{P}\left(\mathcal{S}(x)\leq y\right) - \Phi\left(\frac{y}{\sqrt{\mathcal{U}(x)}}\right)\vert
\leq \sup_{x\in[r,s]}\mathbf{P}\left(\vert \mathcal{S}(x) - \widetilde{\mathcal{M}}_m(\frac{x+m}{2m}) - \widetilde{\mathcal{M}}_m^-(\frac{-x+m}{2m})\vert>\delta\right)\\
+3\sup_{x\in[r,s],y\in\mathbf{R}}\vert \mathbf{P}\left(\widetilde{\mathcal{M}}_m(\frac{x+m}{2m}) - \widetilde{\mathcal{M}}_m^-(\frac{-x+m}{2m})\leq y\right) -  \Phi\left(\frac{y}{\sqrt{\mathcal{U}(x)}}\right)\vert + \sup_{x\in[r,s],y\in\mathbf{R}}\left(\Phi\left(\frac{y+\delta}{\sqrt{\mathcal{U}(x)}}\right) - \Phi\left(\frac{y-\delta}{\sqrt{\mathcal{U}(x)}}\right)\right)
\end{aligned}
\end{equation}
From assumption 4., we prove \eqref{KeyThm}.
\end{proof}

\section{Proofs of theorems in section \ref{chp3}}
\label{appendixB}
The Wasserstein distance can be used to quantify the difference between two probability distributions. We refer chapter 6, \cite{optimalTransport} for a detail introduction. Lemma \ref{Wassiter} bounds the Wasserstein distance between the distribution $T(x) = \frac{1}{n}\sum_{i = 1}^n\mathbf{1}_{\epsilon_i -\overline{\epsilon}\leq x},x\in\mathbf{R}$ and $F$. Here $\overline{\epsilon} = \frac{1}{n}\sum_{i = 1}^n\epsilon_i$.
\begin{lemma}
\label{Wassiter}
Suppose assumption 1. and 2., then
\begin{equation}
\lim_{n\to\infty}\inf_{X,Y}\mathbf{E}^*\vert X - Y\vert^2 = 0 \ \text{almost surely}
\label{Wissterm}
\end{equation}
The infimum is taken over all random variables $(X,Y)\in\mathbf{R}^2$ such that $\mathbf{P}^*(X\leq x) = T(x)$, and $\mathbf{P}^*(Y\leq x) = F(x)$.
\end{lemma}
\begin{proof}
From assumption 1), Gilvenko-Cantelli theorem, and the strong law of large number(e.g., theorem 1.13 in \cite{mathStat}),
\begin{equation}
\lim_{n\to\infty}\sup_{x\in\mathbf{R}}\vert T(x) - F(x)\vert\leq \lim_{n\to\infty}\sup_{x\in\mathbf{R}}\vert\frac{1}{n}\sum_{i = 1}^n\mathbf{1}_{\epsilon_i\leq x} - F(x)\vert + \lim_{n\to\infty}\sup_{x\in\mathbf{R}}\vert F(x + \overline{\epsilon}) - F(x)\vert = 0\ \text{almost surely}
\label{almostSure}
\end{equation}
From the strong law of large number, $\lim_{n\to\infty}\int_{\mathbf{R}} x^2dT = \lim_{n\to\infty}\frac{1}{n}\sum_{i = 1}^n\epsilon^2_i - \lim_{n\to\infty}\overline{\epsilon}^2 =\sigma^2$ almost surely. Choose $x_0 = 0$ in definition 6.8, \cite{optimalTransport}; from proposition 5.7, page 112 in \cite{Stochastic} and theorem 6.9, \cite{optimalTransport}, we prove \eqref{Wissterm}.
\end{proof}
Lemma \ref{thmBOOT} ensures that $\widehat{\mathcal{S}}$ has the same asymptotic distribution as $\mathcal{S}$.
\begin{lemma}
Suppose assumption 1. to 4. hold true. Then for any given $0<r<s<\infty, \xi > 0$,
\begin{equation}
\lim_{n\to\infty} \mathbf{P}\left(\sup_{x\in[r,s]}\sup_{y\in\mathbf{R}}\vert \mathbf{P}^*\left(\widehat{S}(x)\leq y\right) - \Phi\left(\frac{y}{\sqrt{\mathcal{U}(x)}}\right)\vert>\xi\right) = 0
\label{resBoot}
\end{equation}
\label{thmBOOT}
\end{lemma}
We define $d_{1-\gamma} = \Phi^{-1}(1-\gamma)$(i.e., $\Phi(d_{1-\gamma}) = 1-\gamma$). For any given $0<r<s<\infty, \xi > 0$, lemma \ref{thmBOOT} implies with probability tending to $1$, $\forall 2\xi<1 - \gamma < 1 - \xi, r\leq x\leq s$,
\begin{equation}
\begin{aligned}
\mathbf{P}^*\left(\widehat{S}(x)\leq \sqrt{\mathcal{U}(x)}d_{1-\gamma-2\xi}\right) - (1-\gamma-2\xi)\leq \xi\Rightarrow d^*_{1-\gamma}(x)\geq \sqrt{\mathcal{U}(x)}d_{1-\gamma-2\xi}\\
\mathbf{P}^*\left(\widehat{S}(x)\leq \sqrt{\mathcal{U}(x)}d_{1-\gamma+\xi}\right) - (1-\gamma+\xi)\geq -\xi\Rightarrow d^*_{1-\gamma}(x)\leq \sqrt{\mathcal{U}(x)}d_{1-\gamma+\xi}
\end{aligned}
\label{gapD}
\end{equation}

\begin{proof}[Proof of lemma \ref{thmBOOT}]
From lemma \ref{Wassiter}, for almost sure $y$, $\forall 1/4>\delta>0$, $\exists N>0$ such that $\forall n\geq N$, we can find a random variable $(e^\star_1,e^\dagger_1)\in\mathbf{R}^2$, $\mathbf{P}^*(e^\star_1\leq x) = T(x)$(defined in lemma \ref{Wassiter}), $\mathbf{P}^*(e_1^\dagger\leq x) = F(x)$, and $\mathbf{E}^*(e_1^\star - e_1^\dagger)^2<\delta^9$. We generate $n$ independent observations $(e^\star_i,e^\dagger_i), i = 1,2,...,n$ having the same distribution as $(e^\star_1,e^\dagger_1)$, and define $e^\star = (e^\star_1, ..., e^\star_n)^T$, $e^\dagger = (e^\dagger_1, ..., e^\dagger_n)^T$. We choose an integer $m>s+1$, and define
\begin{equation}
\begin{aligned}
\widetilde{M}_m^\star(x) = \sqrt{n}\left(F^{'}(x^{'})\left(x_f^T(X^TX)^{-1}X^Te^\star - \frac{1}{n}\sum_{i=1}^ne^\star_i\right)-\frac{1}{n}\sum_{i=1}^n(\mathbf{1}_{e^\star_i\leq x^{'}} - T(x^{'}))\right)\\ \widetilde{M}_m^\dagger(x) = \sqrt{n}\left(F^{'}(x^{'})\left(x_f^T(X^TX)^{-1}X^Te^\dagger - \frac{1}{n}\sum_{i=1}^ne^\dagger_i\right)-\frac{1}{n}\sum_{i=1}^n(\mathbf{1}_{e^\dagger_i\leq x^{'}} - F(x^{'}))\right), x\in[0,1]\ \text{and } x^{'} = 2mx - m
\end{aligned}
\end{equation}
For any given $1/4 > \delta^{'} > 0, x\in(\frac{1}{2}, 1]$,
\begin{equation}
\begin{aligned}
\mathbf{P}^*\left(\vert(\widetilde{M}_m^\star(x) - \widetilde{M}_m^{\star-}(1-x)) - (\widetilde{M}_m^\dagger(x) - \widetilde{M}_m^{\dagger-}(1-x))\vert > 3\delta^{'}\right)\\
\leq \mathbf{P}^*\left(\sqrt{n}\vert F^{'}(x^{'}) - F^{'}(-x^{'})\vert\times\vert x_f^T(X^TX)^{-1}X^T(e^\star - e^\dagger)\vert> \delta^{'}\right)
+\mathbf{P}^*\left(\vert F^{'}(x^{'}) - F^{'}(-x^{'})\vert\times \vert \frac{1}{\sqrt{n}}\sum_{i=1}^n(e_i^\star - e_i^\dagger)\vert>\delta^{'}\right)\\
+ \mathbf{P}^*\left(\frac{1}{\sqrt{n}}\vert\sum_{i=1}^n\mathbf{1}_{e_i^\star\leq x^{'}} -\mathbf{1}_{e_i^\star< -x^{'}}-T(x^{'})+T^{-}(-x^{'})- \mathbf{1}_{e^\dagger_i\leq x^{'}} + \mathbf{1}_{e^\dagger_i<- x^{'}}+F(x^{'}) - F(-x^{'})\vert > \delta^{'}\right)\\
\leq\frac{(F^{'}(x^{'}) - F^{'}(-x^{'}))^2\mathbf{E}^*(e_1^\star - e_1^\dagger)^2}{\delta^{'2}}\times\left(x_f^T\left(\frac{X^TX}{n}\right)^{-1}x_f+1\right)+
\frac{2}{\delta^{'2}}\mathbf{E}^{*}(\mathbf{1}_{e^\star_1\leq x^{'}} - T(x^{'}) - \mathbf{1}_{e^\dagger_1\leq x^{'}} +F(x^{'}))^2\\
+\frac{2}{\delta^{'2}}\mathbf{E}^*(\mathbf{1}_{e^\star_1<-x^{'}} - T^-(-x^{'}) - \mathbf{1}_{e^\dagger_1<-x^{'}} +F(-x^{'}))^2
\end{aligned}
\end{equation}
$\mathbf{E}^{*}(\mathbf{1}_{e^\star_1\leq x^{'}} - T(x^{'}) - \mathbf{1}_{e^\dagger_1\leq x^{'}} +F(x^{'}))^2\leq 2\mathbf{E}^*(\mathbf{1}_{e^\star_1\leq x^{'}} - \mathbf{1}_{e^\dagger_1\leq x^{'}})^2 + 2\sup_{x\in\mathbf{R}}\vert T(x) - F(x)\vert^2$; from \eqref{Deltas}, $\mathbf{E}^*\vert\mathbf{1}_{e^\star_1\leq x^{'}} - \mathbf{1}_{e^\dagger_1\leq x^{'}}\vert\leq \mathbf{P}^*(\vert e^\star_1 - e^\dagger_1\vert>\delta^{'}) + F(x^{'} + \delta^{'}) - F(x^{'} -\delta^{'})\leq \frac{\delta^9}{\delta^{'2}} + \sup_{x\in\mathbf{R}}(F(x) - F(x-2\delta{'}))$. From dominated convergence theorem, $\mathbf{E}^*(\mathbf{1}_{e^\star_1<-x^{'}} - T^-(-x^{'}) - \mathbf{1}_{e^\dagger_1<-x^{'}} +F(-x^{'}))^2 = \lim_{h\to\infty}\mathbf{E}^*(\mathbf{1}_{e^\star_1\leq-x^{'}-\frac{1}{h}} - T(-x^{'}-\frac{1}{h}) - \mathbf{1}_{e^\dagger_1\leq -x^{'}-\frac{1}{h}} +F(-x^{'}-\frac{1}{h}))^2\leq \frac{2\delta^9}{\delta^{'2}} + 2\sup_{x\in\mathbf{R}}(F(x) - F(x-2\delta{'})) +2\sup_{x\in\mathbf{R}}\vert T(x) - F(x)\vert^2$. Therefore, from \eqref{Deltas}, assumption 4), theorem \ref{theoPd}, corollary \ref{coroLink},
\begin{equation}
\begin{aligned}
\sup_{x\in[\frac{r+m}{2m}, \frac{s+m}{2m}], y \in\mathbf{R}}\vert\mathbf{P}^*\left(\widetilde{M}_m^\star(x) - \widetilde{M}_m^{\star-}(1-x)\leq y\right) - \Phi\left(\frac{y}{\sqrt{\mathcal{U}(x^{'})}}\right)\vert\\
\leq \sup_{x\in[\frac{r+m}{2m}, \frac{s+m}{2m}], y \in\mathbf{R}}\mathbf{P}^*\left(\vert(\widetilde{M}_m^\star(x) - \widetilde{M}_m^{\star-}(1-x)) - (\widetilde{M}_m^\dagger(x) - \widetilde{M}_m^{\dagger-}(1-x))\vert > 3\delta^{'}\right)\\
+3\sup_{x\in[\frac{r+m}{2m}, \frac{s+m}{2m}], y \in\mathbf{R}}\vert\mathbf{P}^*\left(\widetilde{M}_m^\dagger(x) - \widetilde{M}_m^{\dagger-}(1-x)\leq y\right)- \Phi\left(\frac{y}{\sqrt{\mathcal{U}(x^{'})}}\right)\vert+ \sup_{x\in[\frac{r+m}{2m}, \frac{s+m}{2m}], y \in\mathbf{R}}\left(\Phi\left(\frac{y+3\delta^{'}}{\sqrt{\mathcal{U}(x^{'})}}\right) - \Phi\left(\frac{y-3\delta^{'}}{\sqrt{\mathcal{U}(x^{'})}}\right)\right)\\
\Rightarrow \lim_{n\to\infty}\sup_{x\in[\frac{r+m}{2m}, \frac{s+m}{2m}], y \in\mathbf{R}}\vert\mathbf{P}^*\left(\widetilde{M}_m^\star(x) - \widetilde{M}_m^{\star-}(1-x)\leq y\right) - \Phi\left(\frac{y}{\sqrt{\mathcal{U}(x^{'})}}\right)\vert = 0\ \text{almost surely}
\end{aligned}
\label{SecHalf}
\end{equation}
We define a random variable $(e^*_1,e^\star_1)\in\mathbf{R}^2$ which has probability mass $1/n$ on $(\widehat{\epsilon}_i, \epsilon_i - \overline{\epsilon}), i = 1,2,...,n$. We generate independent random variables $(e^*_i,e^\star_i), i = 1,2,...,n$ having the same distribution as $(e^*_1,e^\star_1)$, and define $e^* = (e_1^*,...,e^*_n)^T, e^\star = (e^\star_1,...,e^\star_n)^T$. We define the stochastic process $\widetilde{M}_m^*(x) = \sqrt{n}\left(F^{'}(x^{'})\left(x_f^T(X^TX)^{-1}X^Te^* - \frac{1}{n}\sum_{i=1}^ne^*_i\right)-\frac{1}{n}\sum_{i=1}^n(\mathbf{1}_{e^*_i\leq x^{'}} - \widehat{F}(x^{'})\right)$.
\begin{equation}
\begin{aligned}
\mathbf{P}^*\left(\vert \widetilde{M}_m^*(x) - \widetilde{M}_m^{*-}(1 - x) - \widetilde{M}_m^{\star}(x) + \widetilde{M}^{\star-}_m(1-x)\vert > 3\delta^{'}\right)\\
\leq \frac{\vert F^{'}(x^{'}) - F^{'}(-x^{'})\vert^2\mathbf{E}^*(e^*_1 - e_1^\star)^2}{\delta^{'2}}\times\left(x_f^T\left(\frac{X^TX}{n}\right)^{-1}x_f + 1\right)+ \frac{2}{\delta^{'2}}\mathbf{E}^{*}(\mathbf{1}_{e^*_1\leq x^{'}} - \widehat{F}(x^{'}) - \mathbf{1}_{e^\star_1\leq x^{'}} +T(x^{'}))^2\\
+\frac{2}{\delta^{'2}}\mathbf{E}^*(\mathbf{1}_{e^*_1<-x^{'}} - \widehat{F}^-(-x^{'}) - \mathbf{1}_{e^\star_1<-x^{'}} +T^-(-x^{'}))^2
\end{aligned}
\end{equation}
For
\begin{equation}
\begin{aligned}
\mathbf{E}^*(e^*_1 - e^\star_1)^2 = \frac{1}{n}\sum_{i= 1}^n(\widehat{\epsilon}_i - \epsilon_i + \overline{\epsilon})^2 = \frac{1}{n}\sum_{i= 1}^n\left((x_i - \overline{x}_n)^T(\widehat{\beta} - \beta)\right)^2,\ \mathbf{E}\left((x_i - \overline{x}_n)^T(\widehat{\beta} - \beta)\right)^2 = \sigma^2(x_i - \overline{x}_n)^T(X^TX)^{-1}(x_i - \overline{x}_n)
\end{aligned}
\end{equation}
Assumption 3. implies $\mathbf{E}^*(e^*_1 - e^\star_1)^2 = O_p(1/n)$. From assumption 3. and Cauchy inequality, $\widehat{F}(x) = \frac{1}{n}\sum_{i=1}^n\mathbf{1}_{e_i-\overline{e}\leq x + (x_i - \overline{x}_n)^T(\widehat{\beta} -\beta)}\leq T(x + 2M\Vert\widehat{\beta} - \beta\Vert_2)$, and $\widehat{F}(x)\geq T(x - 2M\Vert\widehat{\beta} - \beta\Vert_2)$.
\begin{equation}
\begin{aligned}
\sup_{x\in\mathbf{R}}\vert\widehat{F}(x) - T(x)\vert\leq \sup_{x\in\mathbf{R}}\vert T(x + 2M\Vert\widehat{\beta} - \beta\Vert_2) - T(x)\vert + \sup_{x\in\mathbf{R}}\vert T(x - 2M\Vert\widehat{\beta} - \beta\Vert_2) - T(x)\vert\\
\leq 4\sup_{x\in\mathbf{R}}\vert F(x) - T(x)\vert + 2\sup_{x\in\mathbf{R}}\vert F(x + 2M\Vert\widehat{\beta} - \beta\Vert_2) - F(x)\vert + 2\sup_{x\in\mathbf{R}}\vert F(x - 2M\Vert\widehat{\beta} - \beta\Vert_2) - F(x)\vert
\end{aligned}
\label{Ts}
\end{equation}
Since $\mathbf{E}^{*}(\mathbf{1}_{e^*_1\leq x^{'}} - \widehat{F}(x^{'}) - \mathbf{1}_{e^\star_1\leq x^{'}} +T(x^{'}))^2\leq \frac{2\mathbf{E}^*(e_1^*-e^\star_1)^2}{\delta^{'2}}+ 2\sup_{x\in\mathbf{R}}\left(T(x+\delta^{'}) - T(x -\delta^{'})\right)+2\sup_{x\in\mathbf{R}}\vert\widehat{F}(x) - T(x)\vert^2$; dominated convergence theorem implies $\mathbf{E}^*(\mathbf{1}_{e^*_1<-x^{'}} - \widehat{F}^-(-x^{'}) - \mathbf{1}_{e^\star_1<-x^{'}} +T^-(-x^{'}))^2 = \lim_{h\to\infty}\mathbf{E}^*(\mathbf{1}_{e^*_1\leq -x^{'}-\frac{1}{h}} - \widehat{F}(-x^{'}-\frac{1}{h}) - \mathbf{1}_{e^\star_1\leq -x^{'}-\frac{1}{h}} +T(-x^{'}-\frac{1}{h}))^2\leq \frac{2\mathbf{E}^*(e_1^*-e^\star_1)^2}{\delta^{'2}}+ 2\sup_{x\in\mathbf{R}}\left(T(x+\delta^{'}) - T(x -\delta^{'})\right)+2\sup_{x\in\mathbf{R}}\vert\widehat{F}(x) - T(x)\vert^2$; and
\begin{equation}
\begin{aligned}
\sup_{x\in[\frac{r+m}{2m}, \frac{s+m}{2m}], y \in\mathbf{R}}\vert\mathbf{P}^*\left(\widetilde{M}_m^*(x) - \widetilde{M}_m^{*-}(1 - x)\leq y\right) - \Phi\left(\frac{y}{\sqrt{\mathcal{U}(x^{'})}}\right)\vert\\
\leq \sup_{x\in[\frac{r+m}{2m}, \frac{s+m}{2m}], y \in\mathbf{R}}\mathbf{P}^*\left(\vert(\widetilde{M}_m^*(x) - \widetilde{M}_m^{*-}(1-x)) - (\widetilde{M}_m^\star(x) - \widetilde{M}_m^{\star-}(1-x)\vert > 3\delta^{'}\right)\\
+3\sup_{x\in[\frac{r+m}{2m}, \frac{s+m}{2m}], y \in\mathbf{R}}\vert\mathbf{P}^*\left(\widetilde{M}_m^\star(x) - \widetilde{M}_m^{\star-}(1-x)\leq y\right)- \Phi\left(\frac{y}{\sqrt{\mathcal{U}(x^{'})}}\right)\vert+\sup_{x\in[\frac{r+m}{2m}, \frac{s+m}{2m}], y \in\mathbf{R}}\left(\Phi\left(\frac{y+3\delta^{'}}{\sqrt{\mathcal{U}(x^{'})}}\right) - \Phi\left(\frac{y-3\delta^{'}}{\sqrt{\mathcal{U}(x^{'})}}\right)\right)
\end{aligned}
\label{SecIntermediate}
\end{equation}
\eqref{almostSure}, \eqref{SecHalf}, and \eqref{Ts} imply for $\forall \xi > 0$, $\lim_{n\to\infty}\mathbf{P}\left(\sup_{x\in[\frac{r+m}{2m}, \frac{s+m}{2m}], y \in\mathbf{R}}\vert\mathbf{P}^*\left(\widetilde{M}_m^*(x) - \widetilde{M}_m^{*-}(1 - x)\leq y\right) - \Phi\left(\frac{y}{\sqrt{\mathcal{U}(x^{'})}}\right)\vert>\xi\right) = 0$. Finally, we adopt the notations in lemma \ref{lemmaBoot},
\begin{equation}
\begin{aligned}
\sup_{x\in[0,1]}\vert\widehat{\mathcal{M}}(x^{'}) - \widetilde{M}_m^*(x)\vert\leq \frac{\sqrt{n}\sup_{x\in\mathbf{R}}\vert F^{''}(x)\vert}{2}\times \left(x_f^T(X^TX)^{-1}X^Te^* - \frac{1}{n}\sum_{i=1}^ne^*_i\right)^2 \\
+ \sup_{x\in[0,1]}\vert\widehat{\alpha}(x^{'}+x_f^T(X^TX)^{-1}X^Te^* - \frac{1}{n}\sum_{i=1}^ne^*_i) - \widehat{\alpha}(x^{'})\vert
\end{aligned}
\end{equation}
and
\begin{equation}
\begin{aligned}
\sup_{x\in[r,s]}\vert\widehat{\mathcal{S}}(x) - \left(\widetilde{M}_m^*(\frac{x+m}{2m}) - \widetilde{M}_m^{*-}(\frac{-x+m}{2m})\right)\vert\leq \sup_{x\in[r,s]}\vert\widehat{\mathcal{M}}(x) - \widetilde{M}_m^*(\frac{x+m}{2m})\vert\\
+\sup_{x\in[r,s]}\lim_{h\to\infty}\vert\widehat{\mathcal{M}}(-x-\frac{1}{h}) - \widetilde{M}_m^*(\frac{-x+m}{2m}-\frac{1}{2hm})\vert\leq 2\sup_{x\in[0,1]}\vert\widehat{\mathcal{M}}(x^{'}) - \widetilde{M}_m^*(x)\vert
\end{aligned}
\end{equation}
$\forall \delta^{'}>0$, \eqref{Deltas} implies
\begin{equation}
\begin{aligned}
\sup_{x\in[r,s], y \in\mathbf{R}}\vert \mathbf{P}^*\left(\widehat{S}(x)\leq y\right) - \Phi\left(\frac{y}{\sqrt{\mathcal{U}(x)}}\right)\vert
\leq \mathbf{P}^*\left(\sup_{x\in[0,1]}\vert\widehat{\mathcal{M}}(x^{'}) - \widetilde{M}_m^*(x)\vert>\delta^{'}\right)+ \sup_{x\in[r,s], y \in\mathbf{R}}\left(\Phi\left(\frac{y+2\delta^{'}}{\sqrt{\mathcal{U}(x)}}\right) - \Phi\left(\frac{y-2\delta^{'}}{\sqrt{\mathcal{U}(x)}}\right)\right)\\
+ 3\sup_{x\in[\frac{r+m}{2m}, \frac{s+m}{2m}], y \in\mathbf{R}}\vert\mathbf{P}^*\left(\widetilde{M}_m^*(x) - \widetilde{M}_m^{*-}(1 - x)\leq y\right) - \Phi\left(\frac{y}{\sqrt{\mathcal{U}(x^{'})}}\right)\vert
\end{aligned}
\end{equation}
From lemma \ref{lemmaBoot}, for any given $\xi>0$, $\exists \frac{1}{4}>\delta^{''} >0, N>0$ such that for any $n\geq N$,

\noindent $\mathbf{P}\left(\sup_{x,y\in[-m-1,m+1],\vert x-y\vert<\delta^{''}}\vert\widehat{\alpha}(x) - \widehat{\alpha}(y)\vert > \frac{\delta^{'}}{4}\right)<\xi$. If $\sup_{x,y\in[-m-1,m+1],\vert x-y\vert<\delta^{''}}\vert\widehat{\alpha}(x) - \widehat{\alpha}(y)\vert \leq \frac{\delta^{'}}{4}$, then
\begin{equation}
\begin{aligned}
\mathbf{P}^*\left(\sup_{x\in[0,1]}\vert\widehat{\mathcal{M}}(x^{'}) - \widetilde{M}_m^*(x)\vert>\delta^{'}\right)\leq\mathbf{P}^*\left(\frac{\sqrt{n}\sup_{x\in\mathbf{R}}\vert F^{''}(x)\vert}{2}\times \left(x_f^T(X^TX)^{-1}X^Te^* - \frac{1}{n}\sum_{i=1}^ne^*_i\right)^2>\frac{\delta^{'}}{2}\right)\\
+\mathbf{P}^*\left(\vert x_f^T(X^TX)^{-1}X^Te^* - \frac{1}{n}\sum_{i=1}^ne^*_i\vert\geq\delta^{''}\right)\leq \left(\frac{\sqrt{n}\sup_{x\in\mathbf{R}}\vert F^{''}(x)\vert}{\delta^{'}}+\frac{1}{\delta^{''2}}\right)\mathbf{E}^*\left(x_f^T(X^TX)^{-1}X^Te^* - \frac{1}{n}\sum_{i=1}^ne^*_i\right)^2
\end{aligned}
\end{equation}
Since $\mathbf{E}^*\left(x_f^T(X^TX)^{-1}X^Te^* - \frac{1}{n}\sum_{i=1}^ne^*_i\right)^2\leq \frac{2\widehat{\sigma}^2}{n}\left(x_f^T(X^TX/n)^{-1}x_f + 1\right)$, from \eqref{SecIntermediate} we prove \eqref{resBoot}.
\end{proof}

\begin{lemma}
Suppose assumption 1. to 4., $\forall -\infty<r<s<\infty, \xi>0$, $\exists \delta>0$ such that
\begin{equation}
\limsup_{n\to\infty}\mathbf{P}\left(\sup_{x\in[r,s]}\sqrt{n}\left(G^*(x) - G^*(x - \frac{\delta}{\sqrt{n}})\right)\geq \xi\right)<\xi
\label{resLemmaS}
\end{equation}
\label{lemmaContinuouity}
$G^*$ is defined in section \ref{chp3}.
\end{lemma}

\begin{proof}
We adopt the notations in lemma \ref{lemmaBoot}. By conditioning on $\epsilon^*$,
\begin{equation}
\begin{aligned}
G^*(x) = \mathbf{E}^*\mathbf{P}^*\left(\vert\varepsilon^* - x_f^T(X^TX)^{-1}X^T\epsilon^*\vert\leq x\Big|\epsilon^*\right) = \mathbf{E}^*\widehat{F}(x + x_f^T(X^TX)^{-1}X^T\epsilon^*) - \mathbf{E}^*\widehat{F}^-(-x + x_f^T(X^TX)^{-1}X^T\epsilon^*)\\
= \left(F(x) - F(-x)\right) + \left(\mathbf{E}^*\left(F(x + x_f^T(X^TX)^{-1}X^T\epsilon^*) - F(x)\right)  - \mathbf{E}^*\left(F(-x + x_f^T(X^TX)^{-1}X^T\epsilon^*) - F(-x)\right)\right)\\
+ \frac{\left(\widehat{\alpha}(x)-\widehat{\alpha}^-(-x)\right)}{\sqrt{n}}+ \left(\frac{1}{\sqrt{n}}\mathbf{E}^*\left(\widehat{\alpha}(x + x_f^T(X^TX)^{-1}X^T\epsilon^*) - \widehat{\alpha}(x)\right)
- \frac{1}{\sqrt{n}}\mathbf{E}^*\left(\widehat{\alpha}^{-}(-x + x_f^T(X^TX)^{-1}X^T\epsilon^*) - \widehat{\alpha}^-(-x)\right)\right)
\end{aligned}
\label{Gs}
\end{equation}
Therefore, for $\forall \frac{1}{4}>\delta>0$,
\begin{equation}
\begin{aligned}
\sup_{x\in[r,s]}\sqrt{n}\left(G^*(x) - G^*(x - \frac{\delta}{\sqrt{n}})\right)\leq 2\delta\sup_{x\in[-s-1,s+1]}\vert F^{'}(x)\vert + \frac{2\sup_{x\in\mathbf{R}}\vert F^{''}(x)\vert\widehat{\sigma}^2}{\sqrt{n}}\left(x_f^T\left(\frac{X^TX}{n}\right)^{-1}x_f\right)\\
+ 2\sup_{x,y\in[-s-1,s+1],\vert x - y\vert\leq \frac{\delta}{\sqrt{n}}}\vert\widehat{\alpha}(x) - \widehat{\alpha}(y)\vert + 4\sup_{x\in[-s-1,s+1]}\mathbf{E}^*\vert\widehat{\alpha}(x + x_f^T(X^TX)^{-1}X^T\epsilon^*) - \widehat{\alpha}(x)\vert
\end{aligned}
\label{downHalf}
\end{equation}
From lemma \ref{lemmaBoot} and \eqref{continuityAlpha}, we prove \eqref{resLemmaS}.
\end{proof}
Suppose assumption 1. to 4., from \eqref{Gs}, \eqref{downHalf}, \eqref{Ts}, and \eqref{almostSure},
\begin{equation}
\begin{aligned}
\sup_{x>0}\vert G^*(x) - (F(x) - F(-x))\vert\leq 2\sup_{x\in\mathbf{R}}\vert\widehat{F}(x) - F(x)\vert + \frac{\sup_{x\in\mathbf{R}}\vert F^{''}(x)\vert\widehat{\sigma}^2}{n}\left(x_f^T\left(\frac{X^TX}{n}\right)^{-1}x_f\right)
\end{aligned}
\label{gapC}
\end{equation}
implies $\forall \xi>0$, $\lim_{n\to\infty}\mathbf{P}\left(\sup_{x>0}\vert G^*(x) - (F(x) - F(-x))\vert>\xi\right) = 0$. If $\sup_{x>0}\vert G^*(x) - (F(x) - F(-x))\vert\leq\xi$, by defining $c_{1-\alpha}$ such that $F(c_{1-\alpha}) - F(-c_{1-\alpha}) = 1-\alpha$,
\begin{equation}
G^*(c_{1-\alpha+2\xi})\geq 1-\alpha+\xi,\ G^*(c_{1-\alpha-2\xi})\leq 1-\alpha-\xi\Rightarrow c_{1-\alpha - 2\xi}\leq c^*_{1-\alpha}\leq c_{1-\alpha+2\xi},\ \forall 2\xi<\alpha<1-2\xi
\label{quantileC}
\end{equation}

\begin{proof}[proof of theorem \ref{THMAS}]
We choose $r,s$ in lemma \ref{thmBOOT} as $c_{(1-\alpha)/4}, c_{1-\alpha/4}$, here $F(c_{z}) - F(-c_{z}) = z$, $\forall z\in(0,1)$. From \eqref{gapD}, \eqref{gapC} and \eqref{quantileC}, for sufficiently small $\xi>0$, with probability tending to $1$, $d^*_{1-\gamma}(c^*_{1-\alpha})\leq \sup_{x\in[c_{1-\alpha-2\xi},c_{1-\alpha+2\xi}]}d^*_{1-\gamma}(x)\leq \sup_{x\in[c_{1-\alpha-2\xi},c_{1-\alpha+2\xi}]}\sqrt{\mathcal{U}(x)}d_{1-\gamma+\xi}$; and $d^*_{1-\gamma}(c^*_{1-\alpha})\geq \inf_{x\in[c_{1-\alpha-2\xi},c_{1-\alpha+2\xi}]}d^*_{1-\gamma}(x)\geq \inf_{x\in[c_{1-\alpha-2\xi},c_{1-\alpha+2\xi}]}\sqrt{\mathcal{U}(x)}d_{1-\gamma-2\xi}$, here $d_{1-\gamma} = \Phi^{-1}(1-\gamma)$. We define $\overline{d} = \sup_{x\in[c_{1-\alpha-2\xi},c_{1-\alpha+2\xi}]}\sqrt{\mathcal{U}(x)}d_{1-\gamma+\xi}$, and $\underline{d} = \inf_{x\in[c_{1-\alpha-2\xi},c_{1-\alpha+2\xi}]}\sqrt{\mathcal{U}(x)}d_{1-\gamma-2\xi}$. With probability tending to $1$, $c^*(1-\alpha,1-\gamma)\leq c^*_{1-\alpha + \frac{\overline{d}}{\sqrt{n}}}\leq c_{1-\alpha+\frac{\overline{d}}{\sqrt{n}}+2\xi}$; and $c^*(1-\alpha,1-\gamma)\geq c^*_{1-\alpha+\frac{\underline{d}}{\sqrt{n}}}\geq c_{1-\alpha + \frac{\underline{d}}{\sqrt{n}}-2\xi}$. We define $\overline{c} = c_{1-\alpha+\frac{\overline{d}}{\sqrt{n}}+2\xi}$, and $\underline{c} = c_{1-\alpha + \frac{\underline{d}}{\sqrt{n}}-2\xi}$. From assumption 1) and 4), $c_\alpha$ is continuous in $\alpha\in(0,1)$; and $\mathcal{U}(x)$ is continuous in $\mathbf{R}$. Since
\begin{equation}
\begin{aligned}
\sqrt{n}\left(\mathbf{P}^*\left(\vert y_f - x^T_f\widehat{\beta}\vert\leq c^*(1-\alpha,1-\gamma)\right) - (1-\alpha)\right) = \mathcal{S}(c_{1-\alpha}) + \left(\mathcal{S}(c^*(1-\alpha,1-\gamma)) - \mathcal{S}(c_{1-\alpha})\right)\\
+ \sqrt{n}\left(G^*(c^*(1-\alpha,1-\gamma)) - (1-\alpha + \frac{d^*_{1-\gamma}(c^*_{1-\alpha})}{\sqrt{n}})\right) + \sqrt{\mathcal{U}(c_{1-\alpha})}d_{1-\gamma} + \left(d^*_{1-\gamma}(c^*_{1-\alpha}) - \sqrt{\mathcal{U}(c_{1-\alpha})}d_{1-\gamma}\right)
\end{aligned}
\label{Gst1}
\end{equation}
we choose $r = \underline{c}$ and $s = \overline{c}$ in lemma \ref{lemmaContinuouity}. With probability tending to $1$,
\begin{equation}
\vert\sqrt{n}\left(G^*(c^*(1-\alpha,1-\gamma)) - (1-\alpha + \frac{d^*_{1-\gamma}(c^*_{1-\alpha})}{\sqrt{n}})\right)\vert\leq \sqrt{n}\left(G^*(c^*(1-\alpha,1-\gamma)) - G^*\left(c^*(1-\alpha,1-\gamma) - \frac{1}{n}\right)\right)<\xi
\end{equation}
We choose a positive integer $m > \overline{c} + 1$. From \eqref{S_and_Mm}, theorem \ref{theoPd}, and lemma \ref{lemmaExt}, with probability tending to $1$,
\begin{equation}
\begin{aligned}
\vert \mathcal{S}(c^*(1-\alpha,1-\gamma)) - \mathcal{S}(c_{1-\alpha})\vert\leq \sup_{x\in[\underline{c}, \overline{c}]}\vert\mathcal{S}(x) - \mathcal{S}(c_{1-\alpha})\vert\\
\leq 2\sup_{x\in[\underline{c}, \overline{c}]}\vert\mathcal{S}(x) - \left(\widetilde{\mathcal{M}}_m(\frac{x+m}{2m}) - \widetilde{\mathcal{M}}_m^-(\frac{-x+m}{2m})\right)\vert + 2\sup_{y,z\in[0,1], \vert y - z\vert\leq \frac{\overline{c} - \underline{c}}{2m}}\vert \widetilde{\mathcal{M}}_m(y) - \widetilde{\mathcal{M}}_m(z)\vert\\
\Rightarrow \text{for $\forall \xi>0$, }\limsup_{n\to\infty}\mathbf{P}\left(\vert \mathcal{S}(c^*(1-\alpha,1-\gamma)) - \mathcal{S}(c_{1-\alpha})\vert>\xi\right)<\xi
\end{aligned}
\label{CsHalf}
\end{equation}
For $\mathcal{U}$ is continuous and $\vert d^*_{1-\gamma}(c^*_{1-\alpha}) - \sqrt{\mathcal{U}(c_{1-\alpha})}d_{1-\gamma}\vert\leq \vert\overline{d} - \sqrt{\mathcal{U}(c_{1-\alpha})}d_{1-\gamma}\vert+\vert\underline{d} - \sqrt{\mathcal{U}(c_{1-\alpha})}d_{1-\gamma}\vert$ with probability tending to $1$, we have for $\forall \xi>0$,
\begin{equation}
\lim_{n\to\infty}\mathbf{P}\left(\vert \sqrt{n}\left(\mathbf{P}^*\left(\vert y_f - x^T_f\widehat{\beta}\vert\leq c^*(1-\alpha,1-\gamma)\right) - (1-\alpha)\right) - \left(\mathcal{S}(c_{1-\alpha}) + \sqrt{\mathcal{U}(c_{1-\alpha})}d_{1-\gamma}\right)\vert >\xi\right) = 0
\label{DeltaSS}
\end{equation}
On one hand, from theorem \ref{theoPd}, $\forall \xi>0$, we choose $Z>0$ such that $\Phi\left(\frac{Z}{\sqrt{\mathcal{U}(c_{1-\alpha})}}\right) - \Phi\left(\frac{-Z}{\sqrt{\mathcal{U}(c_{1-\alpha})}}\right)>1-\xi$, we have $\lim_{n\to\infty}\mathbf{P}(\vert\mathcal{S}(c_{1-\alpha})\vert\leq Z) >1-\xi$. On the other hand, for any given $\xi\in\mathbf{R}$,
\begin{equation}
\begin{aligned}
\lim_{n\to\infty}\mathbf{P}\left(\mathcal{S}(c_{1-\alpha}) + \sqrt{\mathcal{U}(c_{1-\alpha})}d_{1-\gamma}+\xi\geq 0\right) = 1 - \Phi\left(-d_{1-\gamma} - \frac{\xi}{\sqrt{\mathcal{U}(c_{1-\alpha})}}\right) = 1 - \Phi\left(d_{\gamma} - \frac{\xi}{\sqrt{\mathcal{U}(c_{1-\alpha})}}\right)
\end{aligned}
\end{equation}
Combine with \eqref{DeltaSS}, we prove theorem \ref{THMAS}.
\end{proof}

\begin{proof}[Proof of corollary \ref{coroF}]
From theorem 10.1 in \cite{Linreg} and assumption 3., $\widehat{r}_i^{'} = \widehat{\epsilon}_i^{'}/(1-h_i)$ with $h_i = x_i^T(X^TX)^{-1}x_i$, and  $\exists C > 0$ such that $h_i\leq C/n$ for $i=1,2,...,n$. From Cauchy inequality, for sufficiently large $n$
\begin{equation}
\begin{aligned}
\widehat{r}_i = \frac{\widehat{\epsilon}_i}{1-h_i} + \frac{1}{n}\sum_{j=1}^n\frac{(h_i-h_j)\widehat{\epsilon}^{'}_j}{(1-h_i)(1-h_j)}\\
\Rightarrow \sum_{i=1}^n (\widehat{r}_i - \widehat{\epsilon}_i)^2\leq \sum_{i=1}^n\frac{2h_i^2\widehat{\epsilon}_i^2}{(1-h_i)^2} + \frac{2}{n^2}\sum_{i=1}^n\sum_{j=1}^n\frac{(h_i-h_j)^2}{(1-h_i)^2(1-h_j)^2}\sum_{j=1}^n\widehat{\epsilon}_j^{'2}\leq \frac{4C^2}{n^2}\sum_{i=1}^n\widehat{\epsilon}_i^2 + \frac{16C^2}{n^2}\sum_{j=1}^n\widehat{\epsilon}_j^{'2}\\
\Rightarrow \mathbf{E}\sum_{i=1}^n (\widehat{r}_i - \widehat{\epsilon}_i)^2 \leq \frac{4C^2}{n^2}\sum_{i=1}^n\mathbf{E}\widehat{\epsilon}_i^2 + \frac{16C^2}{n^2}\sum_{j=1}^n\mathbf{E}\widehat{\epsilon}_j^{'2}\leq \frac{20C^2}{n^2}\times (2n\sigma^2 + 2\sigma^2\sum_{i=1}^nx_i^T(X^TX)^{-1}x_i)
\end{aligned}
\label{DTS}
\end{equation}
We define a random variable $(\epsilon^*_1, r^*_1)\in\mathbf{R}^2$ having probability mass $1/n$ on $(\widehat{\epsilon}_i,\widehat{r}_i), i = 1,2,...,n$. We generate i.i.d. random variables $(\epsilon^*_i, r^*_i), i = 1,2,...,n$ and $(\epsilon^*_f, r^*_f)$ with the same distribution as $(\epsilon^*_1, r^*_1)$. We denote $\epsilon^* = (\epsilon^*_1,...,\epsilon_n^*)^T$, and $r^* = (r^*_1,...,r^*_n)^T$. For any given $0 < r < s < \infty, \xi>0$, we choose $\delta = C/n^{3/4}$ in \eqref{Deltas} with $C$ a constant,
\begin{equation}
\begin{aligned}
\sup_{x\in[r,s]}\vert G^*(x) - \mathcal{G}^*(x)\vert\leq \mathbf{P}^*\left(\vert\vert\epsilon^*_f - x_f^T(X^TX)^{-1}X^T\epsilon^*\vert - \vert r_f^* - x_f^T(X^TX)^{-1}X^Tr^*\vert\vert > \frac{C}{n^{3/4}}\right)\\
+\sup_{x\in[r,s]}\mathbf{P}^*\left(x-\frac{C}{n^{3/4}}<\vert\epsilon^*_f - x_f^T(X^TX)^{-1}X^T\epsilon^*\vert\leq x + \frac{C}{n^{3/4}}\right)
\leq \frac{4\sqrt{n}}{C^2}\sum_{i = 1}^n(\widehat{\epsilon}_i - \widehat{r}_i)^2\\
+ \frac{4x^T_f(X^TX/n)^{-1}x_f}{C^2\sqrt{n}}\times\sum_{i = 1}^n(\widehat{\epsilon}_i - \widehat{r}_i)^2 +\sup_{x\in[r,s]}\left(G^*(x + \frac{C}{n^{3/4}}) - G^*(x - \frac{C}{n^{3/4}})\right)\\
\end{aligned}
\label{Gsn}
\end{equation}
and for sufficiently large $n$,
\begin{equation}
\begin{aligned}
\sup_{x\geq r}\vert\mathcal{G}^*(x) - (F(x) - F(-x))\vert\leq \frac{4\sqrt{n}}{C^2}\sum_{i = 1}^n(\widehat{\epsilon}_i - \widehat{r}_i)^2
+ \frac{4x^T_f(X^TX/n)^{-1}x_f}{C^2\sqrt{n}}\times\sum_{i = 1}^n(\widehat{\epsilon}_i - \widehat{r}_i)^2\\
+ 3\sup_{x>0}\vert G^*(x) - (F(x) - F(-x))\vert + \sup_{x\geq r}\left( F(x+\frac{C}{n^{3/4}}) - F(x - \frac{C}{n^{3/4}})\right) + \sup_{x\geq r}\left(F\left(-x+\frac{C}{n^{3/4}}\right) - F\left(-x-\frac{C}{n^{3/4}}\right)\right)
\end{aligned}
\label{ProbG}
\end{equation}
Lemma \ref{lemmaContinuouity} and \eqref{gapC} imply $\lim_{n\to\infty}\mathbf{P}\left(\sqrt{n}\sup_{x\in[r,s]}\vert G^*(x) - \mathcal{G}^*(x)\vert > \xi\right) = 0$; and

\noindent$\lim_{n\to\infty}\mathbf{P}\left(\sup_{x\geq r}\vert \mathcal{G}^*(x) - (F(x) - F(-x))\vert > \xi\right) = 0$.

We define $\widehat{\mathcal{F}}(x) = \frac{1}{n}\sum_{i = 1}^n\mathbf{1}_{\widehat{r}_i\leq x}$, and $\widehat{\alpha}(x)$ as in lemma \ref{lemmaBoot}. For any given $-\infty<r<s<\infty, \xi>0$, and sufficiently large $n$, lemma \ref{lemmaBoot} implies
\begin{equation}
\begin{aligned}
\sup_{x\in[r,s]}\vert\widehat{\mathcal{F}}(x) - \widehat{F}(x)\vert\leq \frac{1}{n}\sum_{i = 1}^n\mathbf{1}_{\vert\widehat{r}_i - \widehat{\epsilon}_i\vert>\frac{C}{n^{3/4}}} + \sup_{x\in[r,s]}\frac{1}{n}\sum_{i = 1}^n\mathbf{1}_{x - \frac{C}{n^{3/4}}<\widehat{\epsilon}_i\leq x + \frac{C}{n^{3/4}}}\\
\Rightarrow \mathbf{P}\left(\sqrt{n}\sup_{x\in[r,s]}\vert\widehat{\mathcal{F}}(x) - \widehat{F}(x)\vert > \xi\right)\leq \frac{2}{\sqrt{n}\xi}\sum_{i = 1}^n\mathbf{P}\left(\vert\widehat{\epsilon}_i - \widehat{r}_i\vert>\frac{C}{n^{3/4}}\right) + \mathbf{P}\left(\sup_{x\in[r-1,s+1]}\vert\widehat{\alpha}(x) - \widehat{\alpha}(x - \frac{2C}{n^{3/4}})\vert > \frac{\xi}{4}\right)\\
+\mathbf{P}\left(\sup_{x\in[r-1,s+1]} F^{'}(x)\times \frac{2C}{n^{1/4}}>\frac{\xi}{4}\right)\Rightarrow \lim_{n\to\infty}\mathbf{P}\left(\sqrt{n}\sup_{x\in[r,s]}\vert\widehat{\mathcal{F}}(x) - \widehat{F}(x)\vert > \xi\right) = 0
\end{aligned}
\label{DeltaF}
\end{equation}
Here $C$ is an arbitrary large positive constant.

We define $\widehat{\mathcal{M}}$ and $\widehat{\mathcal{S}}$ as in \eqref{JJYYS}; define $\Lambda(z) = x_f^T(X^TX)^{-1}X^Tz-\frac{1}{n}\sum_{i=1}^nz_i, \forall z = (z_1,...,z_n)^T\in\mathbf{R}^n$; and define
\begin{equation}
\widehat{\mathcal{N}}(x) = \sqrt{n}\left(\widehat{\mathcal{F}}\left(x+\Lambda(u^*)\right) - \frac{1}{n}\sum_{i=1}^n\mathbf{1}_{u^*_i\leq x}\right),\ \
\widehat{\mathcal{T}}(x) = \widehat{\mathcal{N}}(x) -  \widehat{\mathcal{N}}^-(-x)
\end{equation}
Here $u^* = (u^*_1,...,u^*_n)^T$ are i.i.d. random variables generated by drawing from $\widehat{r}$ with replacement. For any given $0<r<s<\infty$ and $\xi>0$,
\begin{equation}
\begin{aligned}
\mathbf{P}^*\left(\sup_{x\in[r,s]}\vert\widehat{\mathcal{S}}(x) - \widehat{\mathcal{T}}(x)\vert>4\xi\right)\leq \mathbf{P}^*\left(\sup_{x\in[r,s]}\vert\widehat{\mathcal{M}}(x) - \widehat{\mathcal{N}}(x)\vert>2\xi\right) + \mathbf{P}^*\left(\sup_{x\in[r,s]}\vert\widehat{\mathcal{M}}^-(-x) - \widehat{\mathcal{N}}^-(-x)\vert>2\xi\right)\\
\leq \mathbf{P}^*\left(\sup_{x\in[r,s]}\sqrt{n}\vert\widehat{F}(x + \Lambda(\epsilon^*)) - \widehat{\mathcal{F}}(x + \Lambda(u^*))\vert > \xi\right) + \mathbf{P}^*\left(\sup_{x\in[r,s]}\vert\frac{1}{\sqrt{n}}\sum_{i = 1}^n\mathbf{1}_{e^*_i\leq x} - \frac{1}{\sqrt{n}}\sum_{i = 1}^n\mathbf{1}_{u^*_i\leq x}\vert>\xi\right)\\
+\mathbf{P}^*\left(\sup_{x\in[r/2, s+1]}\sqrt{n}\vert\widehat{F}(-x + \Lambda(\epsilon^*)) - \widehat{\mathcal{F}}(-x + \Lambda(u^*))\vert > \xi\right) + \mathbf{P}^*\left(\sup_{x\in[r/2, s+1]}\vert\frac{1}{\sqrt{n}}\sum_{i = 1}^n\mathbf{1}_{e^*_i\leq -x} - \frac{1}{\sqrt{n}}\sum_{i = 1}^n\mathbf{1}_{u^*_i\leq -x}\vert>\xi\right)\\
\end{aligned}
\end{equation}
If $\sup_{x\in[-s-2, s+2]}\sqrt{n}\vert\widehat{\mathcal{F}}(x) - \widehat{F}(x)\vert<\xi/4$, and $\sup_{x,y\in[-s-2,s+2],\vert x-y\vert<\delta}\vert\widehat{\alpha}(x) - \widehat{\alpha}(y)\vert \leq \xi/8$ with $0<\delta<1/8$,
\begin{equation}
\begin{aligned}
\mathbf{P}^*\left(\sup_{x\in[r,s]}\sqrt{n}\vert\widehat{F}(x + \Lambda(\epsilon^*)) - \widehat{\mathcal{F}}(x + \Lambda(u^*))\vert > \xi\right), \mathbf{P}^*\left(\sup_{x\in[r/2, s+1]}\sqrt{n}\vert\widehat{F}(-x + \Lambda(\epsilon^*)) - \widehat{\mathcal{F}}(-x + \Lambda(u^*))\vert > \xi\right)\\
\leq  \mathbf{P}^*\left(\sup_{x\in[-s-1,s+1]}\sqrt{n}\vert\widehat{F}(x + \Lambda(\epsilon^*)) - \widehat{F}(x + \Lambda(u^*))\vert > \xi/2\right) + \mathbf{P}^*\left(\sup_{x\in[-s-1,s+1]}\sqrt{n}\vert\widehat{F}(x + \Lambda(u^*)) - \widehat{\mathcal{F}}(x + \Lambda(u^*))\vert>\xi/2\right)\\
\leq \mathbf{P}^*\left(\sup_{x\in[-s-1,s+1]}\vert\widehat{\alpha}(x + \Lambda(\epsilon^*)) - \widehat{\alpha}(x + \Lambda(u^*))\vert>\xi/4\right) + \mathbf{P}^*\left(\sup_{x\in[-s-1,s+1]}\sqrt{n}\vert F(x + \Lambda(\epsilon^*)) - F(x + \Lambda(u^*) )\vert > \xi/4\right)\\
+ \mathbf{P}^*\left(\vert\Lambda(u^*)\vert>1\right)\leq 2\mathbf{P}^*(\vert\Lambda(\epsilon^*)\vert > \delta/4) + 3\mathbf{P}^*(\vert\Lambda(u^*)\vert>\delta/4) + \mathbf{P}^*\left(\sup_{x\in[-s-2,s+2]}\sqrt{n}\vert F^{'}(x)\vert\times\vert\Lambda(\epsilon^*) - \Lambda(u^*)\vert>\xi/4\right)
\end{aligned}
\end{equation}
For $\mathbf{E}^*\Lambda(\epsilon^*)^2\leq \frac{2\widehat{\sigma}^2}{n}\left(x_f^T\left(\frac{X^TX}{n}\right)^{-1}x_f + 1\right)$; $\mathbf{E}^*(\Lambda(u^*)-\Lambda(\epsilon^*))^2\leq \frac{2}{n}\left(x_f^T\left(\frac{X^TX}{n}\right)^{-1}x_f + 1\right)\times \frac{1}{n}\sum_{i = 1}^n(\widehat{\epsilon}_i - \widehat{r}_i)^2$, \eqref{SigmaBound}, \eqref{DTS}, \eqref{DeltaF}, and lemma \ref{lemmaBoot} imply $\lim_{n\to\infty}\mathbf{P}\left(\mathbf{P}^*\left(\sup_{x\in[r,s]}\sqrt{n}\vert\widehat{F}(x + \Lambda(\epsilon^*)) - \widehat{\mathcal{F}}(x + \Lambda(u^*))\vert > \xi\right)>\xi\right) = 0$;

\noindent and $\lim_{n\to\infty}\mathbf{P}\left(\mathbf{P}^*\left(\sup_{x\in[r/2, s+1]}\sqrt{n}\vert\widehat{F}(-x + \Lambda(\epsilon^*)) - \widehat{\mathcal{F}}(-x + \Lambda(u^*))\vert > \xi\right)>\xi\right) = 0$. On the other hand, we define $\widetilde{\alpha}^*$ as in lemma \ref{lemmaBoot}; from \eqref{Deltas}, for any $0<\delta<1/4$,
\begin{equation}
\begin{aligned}
\mathbf{P}^*\left(\sup_{x\in[r,s]}\vert\frac{1}{\sqrt{n}}\sum_{i = 1}^n\mathbf{1}_{e^*_i\leq x} - \frac{1}{\sqrt{n}}\sum_{i = 1}^n\mathbf{1}_{u^*_i\leq x}\vert>\xi\right), \mathbf{P}^*\left(\sup_{x\in[r/2, s+1]}\vert\frac{1}{\sqrt{n}}\sum_{i = 1}^n\mathbf{1}_{e^*_i\leq -x} - \frac{1}{\sqrt{n}}\sum_{i = 1}^n\mathbf{1}_{u^*_i\leq -x}\vert>\xi\right)\\
\leq \mathbf{P}^*\left(\frac{1}{\sqrt{n}}\sum_{i = 1}^n\mathbf{1}_{\vert e^*_i - u^*_i\vert>\frac{\delta}{\sqrt{n}}}>\xi/2\right) + \mathbf{P}^*\left(\sup_{x\in[-s-1,s+1]}\frac{1}{\sqrt{n}}\sum_{i = 1}^n\mathbf{1}_{x-\frac{\delta}{\sqrt{n}}<e^*_i\leq x+\frac{\delta}{\sqrt{n}}}>\xi/2\right)\\
\leq \frac{2\sqrt{n}}{\xi}\mathbf{P}^*\left(\vert e^*_1 - u^*_1\vert>\frac{\delta}{\sqrt{n}}\right) + \mathbf{P}^*\left(\sup_{x\in[-s-2, s+2]}\vert\widetilde{\alpha}^*(x) - \widetilde{\alpha}^*(x - \frac{2\delta}{\sqrt{n}})\vert>\xi/4\right)\\
+ \mathbf{P}^*\left(\sup_{x\in[-s-1,s+1]}\sqrt{n}\vert\widehat{F}(x) - \widehat{F}(x- \frac{2\delta}{\sqrt{n}})\vert > \xi/4\right)
\end{aligned}
\end{equation}
Since $\mathbf{P}^*\left(\vert e^*_1 - u^*_1\vert>\frac{\delta}{\sqrt{n}}\right)\leq \frac{\sum_{i = 1}^n(\widehat{\epsilon}_i - \widehat{r}_i)^2}{\delta^2}$, \eqref{DTS} and lemma \ref{lemmaBoot} imply

\noindent$\lim_{n\to\infty}\mathbf{P}\left(\mathbf{P}^*\left(\sup_{x\in[r,s]}\vert\frac{1}{\sqrt{n}}\sum_{i = 1}^n\mathbf{1}_{e^*_i\leq x} - \frac{1}{\sqrt{n}}\sum_{i = 1}^n\mathbf{1}_{u^*_i\leq x}\vert>\xi\right)>\xi\right) = 0$;

\noindent and $\lim_{n\to\infty}\mathbf{P}\left(\mathbf{P}^*\left(\sup_{x\in[r/2, s+1]}\vert\frac{1}{\sqrt{n}}\sum_{i = 1}^n\mathbf{1}_{e^*_i\leq -x} - \frac{1}{\sqrt{n}}\sum_{i = 1}^n\mathbf{1}_{u^*_i\leq -x}\vert>\xi\right)>\xi\right)= 0$. In particular, $\forall \xi>0$,
\begin{equation}
\begin{aligned}
\lim_{n\to\infty}\mathbf{P}\left(\mathbf{P}^*\left(\sup_{x\in[r,s]}\vert\widehat{\mathcal{S}}(x) - \widehat{\mathcal{T}}(x)\vert>\xi\right)>\xi\right) = 0
\end{aligned}
\end{equation}
For $\forall \xi>0$,
\begin{equation}
\begin{aligned}
\sup_{x\in[r,s],y\in\mathbf{R}}\vert\mathbf{P}^*\left(\widehat{\mathcal{T}}(x)\leq y\right) - \Phi\left(\frac{y}{\sqrt{\mathcal{U}(x)}}\right)\vert\leq \mathbf{P}^*\left(\sup_{x\in[r,s]}\vert\widehat{\mathcal{S}}(x) - \widehat{\mathcal{T}}(x)\vert>\xi\right)\\
+3\sup_{x\in[r,s],y\in\mathbf{R}}\vert\mathbf{P}^*\left(\widehat{\mathcal{S}}(x)\leq y\right) - \Phi\left(\frac{y}{\sqrt{\mathcal{U}(x)}}\right)\vert+ \sup_{x\in[r,s],y\in\mathbf{R}}\left(\Phi\left(\frac{y+\xi}{\sqrt{\mathcal{U}(x)}}\right) - \Phi\left(\frac{y-\xi}{\sqrt{\mathcal{U}(x)}}\right)\right)
\end{aligned}
\label{Taus}
\end{equation}
Lemma \ref{thmBOOT} implies $\lim_{n\to\infty}\mathbf{P}\left(\sup_{x\in[r,s],y\in\mathbf{R}}\vert\mathbf{P}^*\left(\widehat{\mathcal{T}}(x)\leq y\right) - \Phi\left(\frac{y}{\sqrt{\mathcal{U}(x)}}\right)\vert>\xi\right) = 0$.

We choose $\frac{1}{8}\min(\alpha,1-\alpha)>\xi>0$. From \eqref{gapD} and \eqref{Taus}, with probability tending to $1$, $\forall 2\xi<1-\gamma<1-\xi, r\leq x\leq s$, $\sqrt{\mathcal{U}(x)}d_{1-\gamma-2\xi}\leq D^*_{1-\gamma}(x)\leq\sqrt{\mathcal{U}(x)}d_{1-\gamma+\xi}$. We define $c_z, z\in(0,1)$ and $\underline{d},\overline{d}$ as in the proof of theorem \ref{THMAS}. We choose $r = c_{(1-\alpha)/8}>0$ in \eqref{ProbG}, with probability tending to 1, $c_{\tau-2\xi}\leq C^*_{\tau}\leq c_{\tau+\xi},\forall (1-\alpha)/8+2\xi<\tau<1-2\xi$. In particular, this implies $c_{1-\alpha-2\xi}\leq C^*_{1-\alpha}\leq c_{1-\alpha+\xi}$, and $\underline{d}\leq D^*_{1-\gamma}(C^*_{1-\alpha})\leq \overline{d}$. We choose $r = c_{(1-\alpha)/8}$ and $s = c_{1-\alpha + 4\xi}$ in \eqref{Gsn} and lemma \ref{lemmaContinuouity}, $C^*(1-\alpha,1-\gamma)\leq C^*_{1-\alpha+\frac{\overline{d}}{\sqrt{n}}}\leq c^*_{1-\alpha+\frac{\overline{d}+2\xi}{\sqrt{n}}}$; and $C^*(1-\alpha,1-\gamma)\geq C^*_{1-\alpha+\frac{\underline{d}}{\sqrt{n}}}\geq c^*_{1-\alpha+\frac{\underline{d}-3\xi}{\sqrt{n}}}$. We define $\mathcal{S}$ and $\mathcal{U}$ as in \eqref{Sxxxx} and \eqref{VarFun}, since
\begin{equation}
\begin{aligned}
\vert \sqrt{n}\left(\mathbf{P}^*\left(\vert y_f - x^T_f\widehat{\beta}\vert\leq C^*(1-\alpha,1-\gamma)\right) - (1-\alpha)\right) - \left(\mathcal{S}(c_{1-\alpha}) + \sqrt{\mathcal{U}(c_{1-\alpha})}d_{1-\gamma}\right)\vert\\
\leq \vert\sqrt{n}\left(\mathbf{P}^*\left(\vert y_f - x^T_f\widehat{\beta}\vert\leq c^*_{1-\alpha+\frac{\overline{d}+2\xi}{\sqrt{n}}}\right) - (1-\alpha)\right) - \left(\mathcal{S}(c_{1-\alpha}) + \sqrt{\mathcal{U}(c_{1-\alpha})}d_{1-\gamma}\right)\vert\\
+\vert\sqrt{n}\left(\mathbf{P}^*\left(\vert y_f - x^T_f\widehat{\beta}\vert\leq c^*_{1-\alpha+\frac{\underline{d}-3\xi}{\sqrt{n}}}\right) - (1-\alpha)\right) - \left(\mathcal{S}(c_{1-\alpha}) + \sqrt{\mathcal{U}(c_{1-\alpha})}d_{1-\gamma}\right)\vert
\end{aligned}
\end{equation}

Replace $c^*(1-\alpha,1-\gamma)$ in \eqref{Gst1} to \eqref{CsHalf} by $c^*_{1-\alpha+\frac{\overline{d}+2\xi}{\sqrt{n}}}$ and $c^*_{1-\alpha+\frac{\underline{d}-3\xi}{\sqrt{n}}}$, and set $\xi\to 0$, we prove \eqref{PrdTheo}.
\end{proof}

\end{document}